\documentclass[reqno,11pt]{amsart}
\usepackage{amsmath,amsthm,amssymb,bm}
\usepackage{datetime}
\usepackage{enumerate}
\usepackage{soul}

\title{Goldberg's Conjecture is true for random multigraphs}
\author{Penny Haxell}
\thanks{Department of Combinatorics and
		Optimization, University of Waterloo, Waterloo, Ontario, Canada N2L
		3G1. pehaxell@uwaterloo.ca. Partially
		supported by NSERC} 
\author{
	Michael Krivelevich}
\thanks{School of Mathematical Sciences,
		Raymond and Beverly Sackler Faculty of Exact Sciences, Tel Aviv
		University, Tel Aviv, 6997801, Israel. Email:
		krivelev@post.tau.ac.il. Partially supported by USA-Israel BSF grant 2014361, and by ISF grant 1261/17.} 
	\author{ Gal Kronenberg}
	\thanks{School of Mathematical Sciences, Raymond and Beverly Sackler Faculty of Exact Sciences, Tel Aviv University, Tel Aviv, 6997801, Israel. Email: galkrone@mail.tau.ac.il.}

\date{\today}
\allowdisplaybreaks

\def\cM{{\mathcal M}}

\def\mnm{{M(n,m)}}
\newcommand{\whp}{w.h.p.\ }
\newcommand{\Bin}{\ensuremath{\textrm{Bin}}}

\theoremstyle{plain}
\newtheorem{theorem}{Theorem}[section]
\newtheorem{lemma}[theorem]{Lemma}
\newtheorem{claim}[theorem]{Claim}

\newtheorem{observation}[theorem]{Observation}
\newtheorem{corollary}[theorem]{Corollary}
\newtheorem{conjecture}[theorem]{Conjecture}

\newtheorem{remark}[theorem]{Remark}

\usepackage{tikz}
\usetikzlibrary{shapes.geometric}
\usetikzlibrary{fit}
\usetikzlibrary{calc}

\long\def\symbolfootnote[#1]#2{\begingroup\def\thefootnote{\fnsymbol{footnote}}
\footnote[#1]{#2}\endgroup}

\begin{document}
\maketitle

\begin{abstract}
In the 70s, Goldberg, and independently Seymour, conjectured that for
any multigraph $G$, the chromatic index $\chi'(G)$ satisfies $\chi'(G)\leq \max \{\Delta(G)+1,
\lceil\rho(G)\rceil\}$, where $\rho(G)=\max \{\frac {e(G[S])}{\lfloor
  |S|/2\rfloor} \mid S\subseteq V \}$. We show that their conjecture
(in a stronger form) is true  for random multigraphs. Let $\mnm$
be the probability space consisting of all loopless multigraphs with
$n$ vertices and $m$ edges, in which $m$ pairs from $[n]$ are chosen
independently at random with repetitions.  
Our result states that, for a given $m:=m(n)$,  $M\sim M(n,m)$ typically satisfies $\chi'(G)=\max\{\Delta(G),\lceil\rho(G)\rceil\}$. In particular, we show that if $n$ is even and $m:=m(n)$, then $\chi'(M)=\Delta(M)$ for a typical $M\sim M(n,m)$. Furthermore, for a fixed $\varepsilon>0$, if $n$ is odd, then a typical $M\sim M(n,m)$ has $\chi'(M)=\Delta(M)$ for $m\leq (1-\varepsilon)n^3\log n$, and $\chi'(M)=\lceil\rho(M)\rceil$ for $m\geq (1+\varepsilon)n^3\log n$. To prove this result, we develop a new structural characterization of multigraphs with chromatic index larger than the maximum degree.  
\end{abstract}


\section{Introduction}\label{intro}

For a (multi)graph $G=(V,E)$, a \textit{$k$-edge-colouring} of $G$ is a
function $c:E\to [k]$ where $[k]=\{1,\ldots,k\}$, such that
$c(e)\not=c(f)$ whenever $e$ and $f$ share a vertex. We 
denote by $\chi'(G)$ the minimum $k$ such that $G$ has a
$k$-edge-colouring. Since no multigraph with a loop has a
$k$-edge-colouring for any $k$, we will assume throughout this paper
that all our multigraphs are loopless.
It is clear that for every multigraph $G$, the maximum degree $\Delta(G)$ of $G$ is a lower
bound on $\chi'(G)$.

For (simple) graphs, Vizing's Theorem 
\cite{vizing64} tells us that $\chi'(G)\leq\Delta(G)+1$ for every
graph $G$. We say that a graph
$G$ is \textit{Class 1} if $\chi'(G)=\Delta(G)$, and \textit{Class 2}
if $\chi'(G)=\Delta(G)+1$. The problem of determining whether an
arbitrary graph is Class 1 is known to be NP-hard  \cite{Holyer}, and
there has been extensive research regarding the conditions under which
a graph is Class 1 or Class 2. One of the tools to attack this problem
is the following theorem, also due to Vizing.

\begin{theorem}[Vizing \cite{vizing64}]\label{vizing}
	If $G$ is a simple graph with maximum degree $\Delta$ such that every cycle of $G$ contains  a vertex of degree less than $\Delta$, then $\chi'(G)=\Delta$. 
\end{theorem} 
In order to distinguish this result from Vizing's
classical theorem that $\chi'(G)\leq\Delta(G)+1$, we will always
refer to this statement as Theorem~\ref{vizing}.

In the case of multigraphs, the picture is less dichotomous. For a
multigraph $G=(V,E)$ let $\mu(e)$ denote the multiplicity of the pair
$e=\{u,v\}\in \binom V2$ in $G$, and let $\mu(G)=\max\{\mu(e)\mid e\in
\binom V2\}$ be the maximum edge multiplicity of $G$. Vizing's Theorem
\cite{vizing64} for multigraphs states that $\chi'(G)\leq \Delta(G)+\mu(G)$. 
In some cases, the chromatic index can be quite far from the maximum degree.
The classical theorem of Shannon \cite{shannon} states that
$\chi'(G)\leq \lfloor 3\Delta/2\rfloor$, and this is best possible as
can be seen from the graph with three vertices and $\lceil
\Delta/2\rceil$ or $\lfloor \Delta/2\rfloor$ edges joining each pair
of vertices. In this graph every two edges share a vertex, so we need
$\lfloor3\Delta/2\rfloor$ colours to colour the edges of $G$.

Since Theorem \ref{vizing} does not apply to  multigraphs, it is natural to try and find another graph theoretic parameter (besides $\Delta$ and $\mu$) connected to the chromatic index of a multigraph. 
A famous conjecture due to Goldberg (1973),  and also Seymour (1979), is based on
the following parameter. For a multigraph $G=(V,E)$ and $S\subseteq
V$,  let $\rho(S):=\frac {e(G[S])}{\lfloor |S|/2\rfloor}$. Define
$\rho(G)=\max \{\rho(S)\mid S\subseteq V \}$. Then $\chi'(G) \geq
\lceil \rho (G) \rceil$, as can be seen from the following 
argument. For any subset $S\subset V(G)$, every matching in
$G[S]$ has size at most $\lfloor|S|/2\rfloor$. Since every colour
class forms a matching, we need at least $\frac
{e(G[S])}{\lfloor|S|/2\rfloor}$ colours to colour $G[S]$ alone. This
observation gives a lower bound of $\max\{\Delta(G),\lceil \rho (G)
\rceil \}$ on the chromatic index of any
multigraph $G$.
 
 The prevalent belief is that the chromatic index of multigraphs
 should  essentially be determined by $\Delta(G)$ and
 $\rho(G)$. Goldberg \cite{Goldberg1} and independently Seymour
 \cite{Seymour} conjectured the following.
 
 \begin{conjecture}\label{conj:GS}
 	Let $G$ be a loopless multigraph. Then $\chi'(G)\leq \max\{\Delta(G)+1,
 	\lceil\rho(G)\rceil\}$.  
 \end{conjecture}
  Goldberg \cite{Goldberg2} even
 conjectured that if $\lceil\rho(G)\rceil\leq \Delta(G)-1$ then
 $\chi'(G)=\Delta(G)$. On the other hand, for the case $\Delta (G)\leq
 \rho (G)$, Kahn proved \cite{Kahn} that $\chi'(G)\leq (1+o(1))\lceil
 \rho (G) \rceil$. It was also independently conjectured by
 Andersen \cite{Andersen} and  Seymour \cite{Seymour2} that
 $\chi'(G)\in \{\Delta(G),\Delta(G)+1, \lceil\rho(G)\rceil\}$. In the
 past few years, there were several improvements for the upper bound
 on $\chi'(G)$ in terms of $\rho(G)$, see, e.g.,
 \cite{ChenYuZang,Goldberg1,Goldberg2,Haxell,Plantholt,Scheide,Seymour,T}. In
 \cite{ChenYuZang}, and independently \cite{Scheide}, it was shown
 that $\chi'(G)\leq \max\{\Delta+\sqrt{\frac \Delta 2},
 \lceil\rho(G)\rceil \}$. A very recent breakthrough by Chen, Gao, Kim,
 Postle, and Shan~\cite{CGKPS} proves the   best known upper bound of
 $\chi'(G)\leq \max\{\Delta+\sqrt[3]{\frac \Delta 2},
 \lceil\rho(G)\rceil \}$. For a more thorough history of this important problem see~\cite{S}.
 
 One of the main questions in the area of edge colouring is to
 understand which multigraphs have an upper bound for the chromatic
 index that matches the trivial lower bound of
 $\max\{\Delta,\lceil\rho\rceil \}$. This leads to a classification of
 multigraphs with respect to the chromatic index. Similarly to the
 graph case, we would like to 	 
 distinguish between multigraphs whose chromatic index is the same as the trivial lower bound, and those that do not have this property. We say that a multigraph $G$ is \textit{first class} if $\chi'(G)= \max\{\Delta(G),\rho(G)\}$, and otherwise it is \textit{second class}. 
  As proposed in~\cite{S}, it seems natural to expect that almost all multigraphs are first class
  (see  p.186 in \cite{S} for further discussion).

In  this paper we will discuss this problem in the setting of
\textit{random graphs} and \textit{random multigraphs}.  
The two most common random graph models are the Erd\H{o}s-R\'enyi
model $G(n,m)$ (which is the probability space consisting of all
graphs with $n$ labeled vertices and $m$ edges, in which $m$ different
elements from $\binom{[n]}{2}$ are chosen uniformly, one by one,
without repetitions), and the binomial random graph model $G(n,p)$
(the probability space consisting of all graphs with $n$ labeled
vertices, where each one of the $\binom{n}{2}$ possible edges is
included independently with probability $p$).   For the most important
case $p=\frac 12$, Erd\H{o}s and Wilson proved \cite{ErdWil} that a
typical $G\sim G(n,\frac 12)$ is  Class 1. They did it by showing that
with high probability (w.h.p.) $G\sim G(n,\frac 12)$ contains a unique
vertex of maximum degree and then by invoking Vizing's Theorem
(Theorem~\ref{vizing}).   Frieze, Jackson, McDiarmid and Reed
\cite{FriezeJackMcReed} strengthened this result for every constant
$p$, $0<p<1$, and showed that the probability that $G\sim G(n,p)$ is
Class 1 tends to one extremely quickly (concretely, is equal to
$1-n^{-\Theta(n)}$).   
%
%

Here we first show that the Erd\H{o}s-Wilson result can be extended to sparse random graphs.

\begin{theorem}\label{thm:Gnp}
	Let $n$ be sufficiently large integer and  $p:=p(n)=o(1)$. Then \whp $G\sim G(n,p)$ is Class 1.
\end{theorem}

A similar proof will lead to the same result for the model $G(n,m)$, where $m=o(n^2)$.

After having dealt with the chromatic index of random graphs, the next
question  to ask is if the situation remains the same in
random multigraphs. In particular, how does the chromatic index of a
random multigraph $G$ behave typically with respect to the key
parameters $\Delta(G)$ and $\rho(G)$? For this, we need first to set
up a probability space for sampling multigraphs.  
A natural way to define a random multigraph model is to allow edge repetitions in the standard random graph  model $G(n,m)$.  
Let $n$ and $m$ be integers. The model $M(n,m)$ is the probability space consisting of all loopless multigraphs with $n$ vertices and $m$ edges, in which $m$ elements from $\binom {[n]}2$ are chosen 
independently at random with repetitions. 

As far as we know, no work has been done to determine the chromatic
index of a random multigraph (in any random multigraph model). For further
discussion on $M(n,m)$ and other models of random multigraphs see Section~\ref{sec:con}.
In this paper, we prove that  the chromatic index of a typical $M\sim M(n,m)$ is either $\Delta(M)$ or $\lceil\rho(M)\rceil$.

\begin{theorem}\label{GSforRM}
	Let $n$ be an integer and $m:=m(n)$. Let $M\sim M(n,m)$.  Then \whp $\chi'(M)=\max \{\Delta(M),\lceil\rho(M)\rceil\}$.
\end{theorem}

 Thus we essentially confirm the informal conjecture stated in \cite{S} (see  p.186), and in fact show that Conjecture~\ref{conj:GS} holds for random multigraphs,  even in a stronger form.
%
%

In order to determine $\chi'$ precisely, we distinguish between two cases, the case where $n$ is even and where $n$ is odd.

\begin{theorem}\label{thm:nEven}
	Let $n$ be an even integer and let $m:=m(n)$. Let $M\sim M(n,m)$, then \whp $\chi'(M)=\Delta(M)$.
\end{theorem}

The following is a corollary of Theorem \ref{GSforRM} (here and later, we denote by $\log n$ the natural logarithm).

\begin{corollary}\label{cor:oddN}
	Let $\varepsilon>0$, and let $M\sim M(n,m)$ where $n$ is an odd integer. Then the following hold.
	\begin{enumerate}
		\item If $m\leq (1-\varepsilon)n^3\log n$ then \whp $\chi'(M)=\Delta(M)$.
		\item If $m\geq (1+\varepsilon)n^3\log n$ then \whp $\chi'(M)=\lceil \rho(M)\rceil$.
	\end{enumerate}
\end{corollary}

For proving the above theorems, we use Vizing's Theorem (Theorem
\ref{vizing}) and  properties of random (multi)graphs. For larger
values of $m$, we make an extensive use of the method of \emph{Tashkinov
trees} 
(see Section \ref{sec:tash}). This method was introduced by Tashkinov
\cite{T}
in 2000 to address Conjecture~\ref{conj:GS}. In particular, we give a new necessary
condition, of independent interest, for a multigraph to have a large
chromatic index in terms of the two largest degrees and the min/max
edge-multiplicities.

\begin{remark}\label{reIntro}
Our results are in fact algorithmic in the sense that there exists an efficient (polynomial time) algorithm that \whp finds an optimal edge-colouring in a random multigraph drawn from $M(n,m)$. We refer the reader to Remark~\ref{re:TashAlg} and Section~\ref{sec:alg} for more details.
\end{remark}

%
%
%
%

\vspace{3mm}
\subsection{Notation and terminology}

For every positive integer $k$ we use $[k]$ to denote the set $\{1,2,\ldots,k\}$. 
We also write $x \in y \pm z$ for $x \in [y - z, y + z]$.
To avoid confusion, we usually use $e$ for an edge in the (multi)graph, and $\rm{e}$ for the Euler's constant.

Our graph-theoretic notation is standard and follows that
of~\cite{West}. In particular, we use the following.


In a multigraph $G=(V,E)$, $V$ is the set of vertices, and $E$ is a
multiset of elements from $\binom {V}2$.  
Let $d_G(v)=|\{e\in E(G) \mid v\in e \}|$
denote the degree of $v$ in $G$ (including multiplicities). We denote by $\Delta(G)$ and
$\delta(G)$ the maximum degree and the minimum degree in $G$,
respectively. 
We let $\mu(G):=\mu_{max}(G)$ be the maximum edge-multiplicity of $G$, and $\mu_{min}:=\mu_{min}(G)$ the minimum edge-multiplicity of $G$.

For  a set of vertices $U \subseteq V(G)$, we denote by $G[U]$
the corresponding vertex-induced subgraph of $G$, we denote by
$E_G[U]$ the edges of $G[U]$, and write $e_G(U)=|E_G[U]|$. 
For two vertex sets $U,W\subseteq V$ we denote by $E_G[U,W]$ all the edges
$e\in E$ with both endpoints in $U\cup W$ for which $e\cap U\neq
\emptyset$ and $e\cap W\neq \emptyset$. Let $e_G(U,W)= |E_G(U,W)|$.
 Often, when there is no risk of ambiguity, we
omit the subscript $G$ in the above notation.

For a graph $G$ and an edge $e\in E(G)$ (respectively, a vertex $v\in V(G)$) we write $G-e$ (respectively, $G-v$) to refer to the subgraph of $G$ obtained by removing the edge $e$ (respectively, the subgraph of $G$ obtained by removing the vertex $v$).
A multigraph $G$ is said to be $k$-{\it critical} if $\chi'(G)=k+1$ and $G-e$ is
$k$-edge-colourable for every edge $e$. 
For an edge colouring $\phi:E\to C$ of $G=(V,E)$, we say that a colour $c$ is \textit{missing} at a vertex $v\in V$ if $v\notin e$ for every $e\in E$ such that $\phi(e)=c$. We say that $c$ is a \textit{partial edge colouring} if the domain of $\phi$ is a (proper) subset of $E$ (that is, not all the edges are coloured according to $\phi$). We
denote by $E(\phi)$ the set of edges of $G$ that are coloured by
$\phi$, and write $|\phi|=|E(\phi)|$.


We assume that $n$ is large enough where needed.  We say that an event holds \emph{with high probability} (w.h.p.) in an underlying probability space if its probability  tends to one as $n$ tends to infinity.  We  sometimes omit floor
and ceiling signs whenever these are not crucial.

\vspace{3mm}
\subsection{Organization of the paper}
 In the next section we present some auxiliary results, definitions and technical preliminaries; in particular, in Section \ref{sec:tash} we present the tool of Tashkinov trees and prove  useful lemmas regarding the chromatic index of a (general) multigraph, and in Section~\ref{sec:Prop} we prove some properties of random graphs and random multigraphs.   In Section~3 we prove Theorem~\ref{thm:Gnp}. In Section 4 we prove Theorems \ref{GSforRM},    \ref{thm:nEven}, and Corollary \ref{cor:oddN}. In Section~\ref{sec:alg} we discuss the algorithmic issues mentioned in Remark~\ref{reIntro}, and in Section~\ref{sec:con} we make some concluding remarks.

\vspace{5mm}
\section{Tools}

\subsection{Sufficient conditions for $k$-colourability and Tashkinov trees}\label{sec:tash}

The aim of this section is to prove the following theorem that gives a sufficient condition (in terms of the min/max edge-multiplicities, the two largest degrees, and the number of edges) for a multigraph to be $k$-colourable. 
Its proof depends heavily on the method of Tashkinov trees.

\begin{theorem}\label{nnlapp}
	Let $G$ be a multigraph with an odd number $n$ of vertices and 
	degree sequence $d_1\geq d_2\geq\ldots\geq d_n$, minimum multiplicity
	$\mu_{min}$ and maximum multiplicity $\mu$. Suppose
	$\chi'(G)>k$. Then one of the following holds.
	\begin{itemize}
		\item[(a)] $k\leq d_1-1$,  
		\item[(b)] $k\leq d_2+2$,
		\item[(c)] ${9\mu-24}>10\mu_{min}$,
		\item[(d)] $|E(G)|>\frac{(n-1)}2k$.
	\end{itemize}
\end{theorem}

First we describe the method of Tashkinov trees
from~\cite{T}. 
Let $G$ be a multigraph, and let $\phi$ be a partial $(\chi'-1)$-edge colouring
of $G$ (i.e. a $(\chi'-1)$- edge colouring of a subgraph of $G$).  Let $\bar{\phi}(w_j)$ be the 
set of colours missing at $w_j$ according to $\phi$. Let $T =
(w_0,e_0,w_1,...,w_q)$ be a sequence of distinct vertices $w_i$ and edges $e_i$
of $G$, such that each $e_i$ has ends $w_{i+1}$ and $w_k$ for some $k\in
\{0,...,i\}$. Note that $T$ is a tree. We say that $T$ is a $\phi$-Tashkinov tree
if $e_0$ is uncoloured, and for all $i > 0$,
$$\phi(e_i) \in \cM_{T_i,\phi},$$
where $T_i := (w_0,...,w_i)$ and $\cM_{T_i,\phi} = \bigcup_{w_j\in T_i}
\bar{\phi}(w_j)$ is the set of colours missing at some vertex of $T_i$. In other words, $T$ is a $\phi$-Tashkinov tree
if its first edge is uncoloured, and after that, each edge is coloured
with a colour that is missing at a previous vertex. 
We say that a set $S$ of vertices 
is $\phi$-{\em elementary} if no colour $\alpha$ is missing at two distinct
vertices of $S$.
The key property of Tashkinov trees, due to Tashkinov~\cite{T}, is
that if $|\phi|$ is maximum then every $\phi$-Tashkinov tree is
$\phi$-elementary. 

In~\cite{S} (Theorem 5.1), Tashkinov's Theorem~\cite{T} is stated as follows. 
\begin{theorem}\label{TashS}
	Let $G$ be a multigraph with $\chi'(G)=k+1$, where
	$k\geq\Delta(G)+1$. Suppose $e_0$ is an edge such that
	$\chi'(G-e_0)=k$, and 
	let $\phi$ be a $k$-edge-colouring of $G-e_0$. If
	$T$ is a $\phi$-Tashkinov tree starting with $e_0$ then $V(T)$ is
	$\phi$-elementary.  
\end{theorem}

Here we will use a slightly nonstandard statement of
Tashkinov's Theorem that is tailored for our purposes.

\begin{theorem}\label{TashT}
	Let $G$ be a multigraph with $\chi'(G)=k+1$, where $k\geq
	\Delta(G)$. Suppose $e_0$ is an 
	edge such that $\chi'(G-e_0)=k$, where the endpoints $x$ and
	$y$ of $e_0$ satisfy $d(x)+d(y)\leq2k-2$, and suppose that
	$d(v)\leq k-1$ for each $v\in V(G)\setminus\{x,y\}$. 
	Let $\phi$ be a $k$-edge-colouring of $G-e_0$. If
	$T$ is a $\phi$-Tashkinov tree starting with $e_0$ then $V(T)$ is
	$\phi$-elementary. 
\end{theorem}

Here we indicate precisely how to modify the proof of
Tashkinov's Theorem as it is written in~\cite{S} to obtain Theorem~\ref{TashT}.
\begin{proof}
	The condition $k\geq\Delta(G)+1$ is used in the proof of Theorem
	5.1 in~\cite{S} only in the following
	places. 
	
	\begin{enumerate}[$(1)$]
		\item To ensure that if $T$ is a path then Kierstead's Theorem
		(Theorem 3.1 in~\cite{S}) 
		applies (p.117 line 11). However, the assumptions of Kierstead's
		Theorem (p.44) only require that 
		$k\geq\Delta(G)$ and that each vertex of $T$ apart from the
		endpoints of $e_0$ has degree less than $k$.
		\item To ensure that every vertex of $G$ is missing at least one
		colour under $\phi$ (p.120 lines 2 and 22).
		\item To ensure Claim (d) (p.117) in the proof of Theorem 5.1 in~\cite{S},
		which states that if $T'$ is a proper initial segment
                of $T$	with at least two vertices then
		there are at least four colours missing and not used on $T'$. Here
		$T'$ is itself $\phi$-elementary by a minimality assumption on $T$.  
	\end{enumerate}
	
	The assumptions in our statement Theorem~\ref{TashT} of Tashkinov's Theorem
	contain the assumptions for Kierstead's Theorem so (1) follows
	immediately. Statement (2) is also immediate for every $v\in
	V(G)\setminus\{x,y\}$, and follows for $x$ and $y$ as well since
	$k\geq\Delta(G)$ and both $x$ and $y$ are incident to the
	uncoloured edge $e_0$. Thus it remains to
	verify that the assumptions of Theorem~\ref{TashT} ensure (3)
	holds. 
	
	Suppose $V(T')=\{w_0,w_1,\ldots,w_{t+1}\}$ where $t\geq 0$. By
	the assumption $d(x)+d(y)\leq2k-2$, and since $e_0$ is
	uncoloured, there are at least
	$k-(d(x)-1)+k-(d(y)-1)\geq2k+2-(2k-2)=4$ colours missing on the first
	two vertices of $T'$ (which are all different, otherwise $e_0$ could be coloured). Each 
	$w_i$ with $i\geq 2$ is missing at least one colour by the assumption
	$d(w_i)\leq k-1$, and since $T'$ is
	$\phi$-elementary all these $t+4$ colours are distinct. The number of
	colours used on $T$ is at most $|E(T)|-1=|V(T)|-2=t$. Hence at
	least four colours are 	missing and not used on $T'$, verifying~(3).
\end{proof}

\begin{remark}\label{re:TashAlg}
	In the proof of Theorem~\ref{TashT} the condition that $\chi'(G)=k+1$ 	can be replaced with the argument that for a
	colouring $\phi$, if $V(T)$ is not $\phi$-elementary then one can
	change the colouring $\phi$ to get a $k$-edge-colouring of
        $G$. The way to change $\phi$ is based on an alternating-paths
        argument and as a result Theorem~\ref{TashT} is in fact
        algorithmic. See Section~\ref{sec:alg} for more details.
\end{remark}

We also need the following simpler statement, which follows easily from known results, e.g.~\cite{CK} (see
also~\cite{S}, Theorem 2.9). It can also be proved directly, with an argument similar to the proof of Vizing's Theorem.  However, as a warm-up we show how to
derive it from Theorem~\ref{TashT}.

\begin{lemma}\label{lem:Chi'Tash}
	Let $k$ be a positive integer. Let $G$ be a multigraph with
        maximum degree at most $k$ and maximum edge 
	multiplicity at most $\mu$. Suppose that every vertex of $G$ except one has
	degree at most $k-t$ where $t\geq\max\{2,\mu\}$. Then
	$\chi'(G)\leq k$. 
\end{lemma}

\begin{proof}
	
	Suppose on the contrary that $\chi'(G)\geq k+1$.
	Note that every
	subgraph of $G$ also satisfies the conditions of the lemma.
	Thus by removing edges one by one if
	necessary, we may assume that $\chi'(G)=k+1$ but
	$\chi'(G-e)=k$ for every edge $e$, i.e. $G$ is
	$k$-critical. 

The conditions tell us that $G$ has at most one vertex of degree
larger than $k-t$. If such a vertex exists, let us call it $x$ and
choose an edge $e_0$ incident to $x$ in $G$ 
	to be the uncoloured edge. Otherwise we can choose $e_0$ to be
	an arbitrary edge and set $x$ to be one of its endpoints.
	
	Let $\phi$ be a $k$-edge-colouring of $G-e_0$.  Let $y$ denote the other
	endpoint of $e_0$. Then the degree assumption implies that
	$d(x)+d(y)\leq k+k-t\leq2k-2$ and $d(v)\leq k-1$ for 
	each $v\in V(G)\setminus\{x,y\}$.
	Let $T$ be a maximal $\phi$-Tashkinov tree starting with
	$e_0$. Thus the assumptions of Theorem~\ref{TashT} are satisfied,
	and hence $T$ is $\phi$-elementary.
	
	Clearly $|V(T)|\geq 2$.  
	Note that since every vertex different from $x$ has
	degree at most $k-t$, each has at least $t$ missing colours. Since
	also $y\in V(T-x)$ is incident to the uncoloured edge $e_0$ in
	$G[(V(T)]$, the total number of colours missing at vertices of $T-x$ is
	at least $t(|V(T)|-1)+1$. Since $T$ is $\phi$-elementary these are
	all distinct.
	
	Since $T$ is maximal we may conclude that for every colour
	$\alpha\in\cM_{T,\phi}$, no edge $e$ joining $V(T)$ to
	$V(G)\setminus V(T)$ is such that $\phi(e)=\alpha$.
	Therefore since $T$ is $\phi$-elementary, all colours in
	$\cM_{T-x,\phi}$ appear 
	on edges of $G[(V(T)]$ incident to $x$. Therefore $d_{V(T)}(x)\geq
	t(|V(T)|-1)+1$. But 
	clearly $d_{V(T)}(x)\leq\mu(|V(T)|-1)$, so we find $\mu>t$, contradicting
	the assumption of the lemma. Hence $\chi'(G)\leq k$.
\end{proof}
We apply Lemma~\ref{lem:Chi'Tash} only in the case $k=\Delta(G)$,
but for the purposes of the discussion of algorithmic issues in Section~\ref{sec:alg} it is
more convenient to state it in the more general form given here.

Using the previous lemma, we can deduce the following upper bound for the chromatic index,
 assuming that the difference between the two largest degrees is bigger than the difference between the maximum and the minimum edge multiplicity.  

\begin{corollary}\label{cor:RemoveCopiesKn}
	Let $G$ be a multigraph on $n$ vertices with
	degree sequence $d_1\geq d_2\geq\ldots\geq d_n$. Assume that $d_1-d_2\geq \mu(G)-\mu_{min}\geq 2$, where $\mu_{min}:=\mu_{min}(G)$ is the minimum multiplicity of $G$. Then $\chi'(G)=\Delta(G)$ if $n$ is even and $\chi'(G)\leq\Delta(G)+\mu_{min}$ if $n$ is odd.
\end{corollary}

\begin{proof}
	We decompose the graph $G$ into two multigraphs: $G_1\cup G_2=G$. Let $G_1$ be the complete multigraph on $n$ vertices with all edge multiplicities $\mu_{min}$. If $n$ is even, we have that $\chi'(K_n)=n-1$ and therefore $\chi'(G_1)= (n-1)\mu_{min}$. If $n$ is odd, we have that $\chi'(K_n)=n$ and therefore $\chi'(G_1)\leq n\mu_{min}$.
	
	Now let $E(G_2)=E(G)\setminus E(G_1)$ and $V(G_2)=V(G)$. Then $\mu(G_2)=\mu(G)-\mu_{min}$. Denote by $d_1(G_2), d_2(G_2)$ the two largest degrees in $G_2$. Then $d_1-d_2=d_1(G_2)-d_2(G_2)$.
	By the assumption we have that $d_1-d_2\geq \mu(G)-\mu_{min}$, that is, 
	$\mu(G_2)\leq d_1(G_2)-d_2(G_2)$. Thus, using Lemma \ref{lem:Chi'Tash}
	applied on $G_2$ with $k=\Delta(G_2)$ we  have that  $\chi'(G_2)= \Delta(G_2)=\Delta(G)-(n-1)\mu_{min}$.
	
	All in all, we have that $$\chi'(G)\leq \chi'(G_1)+\chi'(G_2)\leq n\mu_{min}+\Delta(G)-(n-1)\mu_{min}=\Delta(G)+\mu_{min},$$ when $n$ is odd, and
	$$\chi'(G)\leq \chi'(G_1)+\chi'(G_2)=(n-1)\mu_{min}+\Delta(G)-(n-1)\mu_{min}=\Delta(G),$$ when $n$ is even. Since always $\Delta(G)\leq \chi'(G)$, the claim follows.   
\end{proof}

The remainder of this section is devoted to the proof of Theorem~\ref{nnlapp}.

\subsubsection{The structure of maximum Tashkinov trees}

In order to prove Theorem \ref{nnlapp} we use the following lemma
which gives a more detailed picture of the structure of a multigraph
with a  maximum Tashkinov tree $T$. The set $Q$ will be $V(T)$, and
the other sets will be described in the discussion following the
statement of the lemma.

\begin{lemma} \label{nnl} Let $G$ be a multigraph with
	degree sequence $d_1\geq d_2\geq\ldots\geq d_n$. Suppose $\chi'(G)>k$
	where $k\geq d_1$ and $k\geq d_2+2$. Then there exist vertex sets $X$, $Z$,
	$Q$ and $U$ in $G$ with the following properties.
	\begin{enumerate}[$(1)$]
		\item $Z\subset Q\subseteq U$ and $X=V(G)\setminus U$,
		\item $|Q|$ is odd and $|Q|\geq 2+\sum_{z\in Z}(k-d_G(z))$,
		\item 
		$$|E[Q]|>\frac{|Q|-1}2(|E[X,Q\setminus Z]|+\sum_{v\in
			Q}(k-d_G(v))+2),$$
		\item $\sum_{v\in U}(k-d_G(v))\leq k-2$. Moreover if $n$ is odd and
		$|X|\leq 1$ then $\sum_{v\in V(G)}(k-d_G(v))\leq k-2$.
	\end{enumerate}
\end{lemma}

The rest of this subsection is devoted to the proof of
Lemma~\ref{nnl}. We begin by recalling some definitions. 
For a partial edge colouring $\phi$ of a multigraph $G$,
a set $S$ of vertices is said to be $\phi$-{\it elementary} if no two
vertices in $S$ have a common missing colour under $\phi$. A colour
$\beta$ is
said to be {\it used} on a Tashkinov tree $T$ if some edge of $T$ is coloured
$\beta$. (Note that any used colour is necessarily missing at some
vertex of $T$.) A colour $\gamma$ is said to be {\it defective} for
$T$ if at least two edges coloured
$\gamma$ leave $V(T)$. We write $D_T$ for the set of defective colours
for $T$. For a set 
$S$ and a colour $\gamma$ we write $\Gamma_{\gamma}(S)$ for the set
of vertices outside $S$ that are joined to $S$ via an edge coloured
$\gamma$. A {\it max-pair} for $G$ is
a pair $(\phi,T)$ such that $\phi$ is a partial edge colouring with $|\phi|$  as large as possible, $T$ 
is a $\phi$-Tashkinov tree, and for all pairs $(\psi,T')$ where $\psi$
is a partial colouring with the same set of colours as $\phi$,
$E(\psi)=E(\phi)$, and $T'$
is a $\psi$-Tashkinov tree, we have $|V(T)|\geq|V(T')|$.   

Let $G$ be as in the assumptions of Lemma~\ref{nnl}. It suffices to
prove the statement for $k=\chi'(G)-1$. 

Next we show that it suffices to prove Lemma~\ref{nnl} for
$k$-critical multigraphs $G$. To see this, first note that if $G'$ is
a spanning subgraph of $G$ then the assumptions $k\geq d_1$ and
$k\geq d_2+2$ still hold for $G'$. Suppose we are able to prove that there
exist sets $Q$, $Z$, $U$ and $X$ such that the
conclusions of Lemma~\ref{nnl} hold for $G'$. Then it is clear that
Conclusions (2) and (4) hold for $G$ as well, since $d_G(v)\geq
d_{G'}(v)$ for every vertex $v$. Conclusion (1) is the same for $G$ and $G'$
since $V(G)=V(G')$. For Conclusion (3), let $t=|E_G[X,Q\setminus
Z]|-|E_{G'}[X,Q\setminus Z]|$. Then $\sum_{v\in
	Q}d_G(v)\geq\sum_{v\in Q}d_{G'}(v)+t$, and therefore
\begin{align*}
|E_G[Q]|&\geq|E_{G'}[Q]|>\frac{|Q|-1}2(|E_{G'}[X,Q\setminus Z]|+\sum_{v\in
	Q}(k-d_{G'}(v))+2)\\
&\geq\frac{|Q|-1}2(|E_G[X,Q\setminus Z]|-t+\sum_{v\in
	Q}(k-d_{G}(v))+t+2)\\
&= \frac{|Q|-1}2(|E_G[X,Q\setminus Z]|+\sum_{v\in
	Q}(k-d_{G}(v))+2).
\end{align*}
This shows that we may assume (by removing edges of $G$ one by one if
necessary) that $G$ is $k$-critical. As before, if there still
exists a vertex of degree larger than $k-2$ (which will be unique if it
exists by the assumptions of the lemma) then we choose one of its incident edges to be the uncoloured
edge $e_0$. Otherwise $e_0$ can be chosen arbitrarily. Then the
assumptions of Lemma~\ref{nnl} imply 
that the endpoints $x$ and $y$ of $e_0$ satisfy
$d(x)+d(y)\leq d_1+d_2\leq 2k-2$, and that $d(v)\leq k-1$ for each $v\in
V(G)\setminus\{x,y\}$. Thus at any point in the proof we may apply
Theorem~\ref{TashT} to any $k$-edge-colouring $\phi$ of $G-e_0$ and
any $\phi$-Tashkinov tree $T$.

We choose $\phi$ and $T$ such that $(\phi,T)$ is a max-pair starting with $e_0$.
The sets in the conclusion of Lemma~\ref{nnl} are defined as follows:
$Q=V(T)$, $Z\subset Q$ is the set of vertices $v$ in $Q$ such that
every colour missing at $v$ is used on $T$ (note $Z=\emptyset$ is
possible), $U=Q\cup(\bigcup_{\gamma\in
	D_T}\Gamma_{\gamma}(Q\setminus Z))$, and $X=V(G)\setminus U$.

Our proof will consist of a series of lemmas. We remark that many
of the results and ideas used here have appeared in other works
(e.g. \cite{ChenYuZang,Scheide,S,T}) but in order to keep this paper self-contained we
include all proofs.  

\begin{lemma}\label{q} With these definitions,
	$|Q|$ is odd and $|Q|\geq 2+\sum_{z\in Z}(k-d_G(z))$. Moreover
	every defective colour for $T$ occurs on at least three edges 
	leaving $Q$.
\end{lemma}

\begin{proof}
	Since $Q$ contains the two endpoints of the uncoloured edge we know
	$|Q|\geq 2$. Since $k\geq d_1$, both endpoints of the uncoloured edge
	are missing at 
	least one colour (and by maximality of $\phi$ no colour is missing at
	both endpoints), thus by 
	definition of Tashkinov tree and by maximality $|Q|>2$. 
	By maximality of $T$, no edge coloured with a missing colour $\beta$ leaves
	$Q$, so by Theorem~\ref{TashT} the $\beta$ colour class in $G[Q]$ is a
	matching that misses 
	exactly one vertex, i.e. $|Q|$ is odd (and hence at least
	three). By maximality of $T$, no defective colour is missing on $T$, which
	implies that every defective colour occurs on at least three edges
	leaving $V(T)$.
	
	Each vertex  $z\in Z$ is missing
	at least $k-d_G(z)$ colours, and by Theorem~\ref{TashT} these are all
	distinct. Hence the number of colours appearing on edges of $T$ is at
	least $\sum_{z\in Z}(k-d_G(z))$. One edge of $T$ is uncoloured, thus
	$T$ has at least $\sum_{z\in Z}(k-d_G(z))+2$ vertices. 
	%
	%
\end{proof}

\begin{lemma}\label{eq} The set $Q$ satisfies
	$$|E[Q]|>\frac{|Q|-1}2(|E[X,Q\setminus Z]|+\sum_{v\in
		Q}(k-d_G(v))+2).$$
\end{lemma}

\begin{proof}
	By definition of $U$, no edge joining $Q\setminus Z$ to $X$ is
	coloured with a defective colour (for $Q$). By maximality of $T$, 
	no edge joining $Q\setminus Z$ to $X$ is
	coloured with a colour missing on $T$. Recall also that the only
	uncoloured edge is $e_0$ which is not in $E[Q\setminus Z,X]$. Thus each edge of
	$E[Q\setminus Z,X]$ is coloured with some $\delta$ where the colour
	class of $\delta$ in 
	$G[Q]$ is a matching of size $\frac{|Q|-1}2$, and these matchings are
	all disjoint from the colour classes of missing colours (which are
	also matchings of size $\frac{|Q|-1}2$). Moreover, no colour $\delta$
	can appear more than once on $E[Q\setminus Z,X]$, otherwise it would
	be a defective colour.
	
	The number of colours  missing on $T$ is at least $\sum_{v\in
		Q}(k-d_G(v))+2$, and these are all distinct by
	Theorem~\ref{TashT}. Putting these facts together gives that the
	number of coloured edges in $E[Q]$ is at least
	$$\frac{|Q|-1}2(|E[X,Q\setminus Z]|+\sum_{v\in
		Q}(k-d_G(v))+2).$$
	The fact that the uncoloured edge $e_0$ is also in $E[Q]$ gives the
	strict inequality, as required.
\end{proof}

The next lemma will imply Conclusion (4)  of
Lemma~\ref{nnl} immediately.

\begin{lemma}\label{u}
	The set $U$ is $\phi$-elementary.
\end{lemma}

\begin{proof} 
	Let $qa$ be an edge of $G$ with $q\in Q\setminus Z$ and
	$a\notin Q$. Suppose $qa$ is coloured with some $\gamma$ 
	that is defective for
	$T$. Since $q\notin Z$ we know that there exists a colour
	$\sigma$ missing at $q$ that is not
	used on $T$. 
	
	\noindent {\bf Claim 1:} Let $\alpha$ be a colour
	missing at $a$. Then $\alpha$ is not missing on $T$.
	
	\noindent{\it Proof of Claim 1.} Suppose on the contrary that $\alpha$
	is  missing on $T$. By maximality of $T$ no
	edge coloured $\alpha$ or $\sigma$ leaves $T$. Hence the
	$(\alpha,\sigma)$-path $P$ from $a$ does not contain any vertex of $Q$.
	Form a new colouring by switching on $P$. Then $\sigma$ is missing at
	both $a$ and 
	$q$, and the tree $T$ is unchanged. Now
	recolour $qa$ from $\gamma$ to 
	$\sigma$, to get a colouring $\psi$ with $E(\psi)=E(\phi)$. Then
	$T$ is a $\psi$-Tashkinov tree since $\sigma$ was not used on $T$. Note
	$\gamma\notin\{\alpha,\sigma\}$. Since
	$\gamma$ was defective, there remain at least two edges leaving $Q$
	coloured $\gamma$, while $\gamma$ is missing at $q$. Hence $T$ could
	be extended under $\psi$, contradicting that fact that $(\phi,T)$ was 
	a max-pair.
	
	\noindent {\bf Claim 2:} Let $\alpha$ be a colour missing at $a$. Then
	the $(\alpha,\sigma)$-path $P$ from $a$ ends at $q$.
	
	\noindent{\it Proof of Claim 2.} Suppose on the contrary that $P$ does
	not end at $q$. Switch on $P$ to obtain a new colouring. Then
	since no edge of $T$ was coloured $\alpha$ (using Claim 1) or
	$\sigma$, we know that 
	$T$ is unchanged. But now $\sigma$ is missing at $a$ as well, so recolouring
	$qa$ to $\sigma$ gives a colouring $\psi$ that again contradicts that fact
	that $(\phi,T)$ was a max-pair, because $T$ is still a
	$\psi$-Tashkinov tree and (as in Claim 1) we can extend
	$T$ via another edge coloured $\gamma$.
	
	\medskip
	
	Now consider another edge $rb$ of $G$ with $r\in Q\setminus Z$ and
	$b\notin Q$, where $a\not=b$ ($q=r$ is possible). Suppose $rb$ is
	coloured with some
	$\delta$ that is defective for $T$. (Note $\gamma=\delta$
	is possible, in which case $q\not= r$.)
	
	\noindent {\bf Claim 3:} No colour $\alpha$ is missing at both $a$ and
	$b$.
	
	\noindent{\it Proof of Claim 3.} Suppose on the contrary that $\alpha$
	is such a colour. We know by Claim 2 that the $(\alpha,\sigma)$-path
	$P$ from $a$ ends at $q$. Switch on $P$ to obtain a new colouring
	$\psi$. Since no edge of $T$ was coloured $\alpha$ (using Claim 1) or
	$\sigma$, we know that $T$ is unchanged, except that $\alpha$ is now
	missing at $q$ instead of $\sigma$. The vertex $b$ was not on $P$ and
	hence $\alpha$ is still missing at $b$. But now $(\psi,T)$ is a
	max-pair and $\delta$ is a defective colour for $T$ (observe
	$\delta\notin\{\alpha,\sigma\}$), and $\alpha$ is missing at $b$ and
	also at $q\in Q$. This contradicts Claim 1 applied to $rb$ and the
	max-pair $(\psi,T)$.   
	
	\medskip
	
	Now Claims 1 and 3 together with Theorem~\ref{TashT} give us the
	desired conclusion that $U$ is $\phi$-elementary. 
\end{proof}

Lemma~\ref{u} implies Conclusion (4) in Lemma~\ref{nnl} because
there are at least $2+\sum_{v\in U}(k-d_G(v))$ missing colours on the
vertices of $U$ (recall that $Q\subseteq U$ contains the
uncoloured edge). This is therefore at most the total number $k$ of
colours. Moreover if $n$ is odd then no colour can be missing on
exactly two vertices of $G$, hence if $|X|\leq 1$ then in fact the
whole set $V(G)$ is $\phi$-elementary. This implies that $2+\sum_{v\in
	V(G)}(k-d_G(v))\leq k$ as required.





\subsubsection{Conditions for high chromatic index}

In this section we show how to use Lemma \ref{nnl} in order to prove Theorem~\ref{nnlapp}.

%

\begin{proof}
	Let $G$ be as given with $\chi'(G)>k$, and suppose that
	none of (a)--(c) holds. Then in particular the assumptions of Lemma~\ref{nnl}
	are satisfied. Let $X$, $Z$, $Q$ and $U$ be as given by
	Lemma~\ref{nnl}.

	First we claim that $|Z|\leq\frac{|Q|+1}3$. By Lemma~\ref{nnl}(2) we know
	$|Q|\geq 2+\sum_{z\in Z}(k-d_G(z))$. Since (a) and
	(b) do not hold we find $|Q|\geq 2+3(|Z|-1)$, verifying our claim.
	
	Since
	$|E[X,Q\setminus Z]|\geq|X|(|Q|-|Z|)\mu_{min}$, we derive from 
	Lemma~\ref{nnl}(3) that 
	$$\mu\geq\frac{|E[Q]|}{{|Q|\choose
			2}}>|X|(1-\frac{|Z|}{|Q|})\mu_{min}+\frac1{|Q|}(2+\sum_{v\in Q}(k-d_G(v))).$$
	Since (a) and (b) do not hold we know that $\sum_{v\in
		Q}(k-d_G(v))\geq 3(|Q|-1)$. Therefore
	$\mu>|X|(1-\frac{|Z|}{|Q|})\mu_{min}+3-\frac1{|Q|}$. If $\mu_{min}=0$, then by the fact that $(c)$ does not hold we have that $\mu\leq 2$. 	By Lemma~\ref{nnl}(2) we know that $|Q|\geq 3$ because $|Q|$ is odd, and therefore 	$\mu>3-\frac1{|Q|}>2$. Thus we can assume that $\mu_{min}> 0$.
	Since
	$|Z|\leq\frac{|Q|+1}3$ we find 
	$$|X|<\frac{\mu-3+\frac1{|Q|}}{\mu_{min}(1-\frac{|Z|}{|Q|})}=\frac{\mu-3+\frac1{|Q|}}{\mu_{min}}\Big{(}\frac3{2-\frac1{|Q|}}\Big{)}.$$
	Again, by the fact that $|Q|\geq 3$  we conclude (using the fact that (c) does not hold) that 
	$|X|<\frac{\mu-3+\frac1{3}}{\mu_{min}}(\frac3{2-\frac1{3}})\leq2$. 
	
	Therefore $|X|\leq 1$. Since $n$ is odd, Lemma~\ref{nnl}(4) tells us
	that $\sum_{v\in V(G)}(k-d_G(v))\leq k-2$, in other words
	$kn-2|E(G)|\leq k-2$. Thus (d)
	holds, completing the proof. 
\end{proof}

\vspace{3mm}
\subsection{Binomial distribution bounds}\label{sec:Bin}

In many of the probabilistic statements in this paper we use the \textit{binomial distribution}. We say that $X\sim \Bin(m,p)$ if $\Pr[X=k]=\binom{m}{k}p^k(1-p)^{m-k}$ where $0\leq k\leq m$ is an integer. 
We start with four basic observations.

\begin{observation}\label{obs:BinDec}
	Let $d,d',m,m'$ be integers and let $p,p'\in [0,1]$. Assume that $m'\leq m$, $d\leq d'$ and $p'\leq p$. Then for $X\sim \Bin(m,p)$ and $X'\sim \Bin (m',p')$ we have $\Pr[X'\geq d']\leq \Pr[X\geq d]$.
\end{observation}


%

%

\begin{observation}\label{obs:BinFrac}
	Let $0<p<1$ and let $k,m$ be  integers, $0<k\leq \frac m2$. Let $X\sim \Bin(m,p)$. Then
	$$\frac {\Pr[X=k]}{\Pr[X=k+1]}=\frac {(k+1)(1-p)}{(m-k)\cdot p}<\frac {4k}{mp}. $$
	In addition, $\Pr[X\geq k]=\Pr[X=k]+\Pr[X\geq k+1]<(1+\frac {4k}{mp})\Pr[X\geq k+1]$.
\end{observation}

In the following claim we show that a binomial variable decays relatively slowly around its expectation.

\begin{claim}\label{claim:alphaBin}
	Let $p:=p(m)=o(1)$
	and $mp=\omega(1)$. Let  $h\geq1$ be such that $d=mp+h$ is an integer and $h=o(mp)$. Then for $X\sim \Bin(m,p)$  and some $\alpha=\Theta\left(\frac {h}{mp}\right)$ we have
	$\Pr[X=d+1]=(1-\alpha)\Pr[X=d]$.
	If also $h=\Omega(\sqrt {mp})$, then 
	$$\Pr[X=mp+h]=O\left(\alpha\right)\Pr[X\geq mp+h].$$
	\end{claim}
	
	\begin{proof}
		For the first part,
			\begin{align*}
			\frac {\Pr[X=d+1]}{\Pr[X=d]} &= \frac {m-d}{d+1}\cdot  \frac p{1-p}
			= \frac {m-mp-h}{mp+h+1}\cdot  \frac p{1-p}\\
			&= \frac {mp-mp^2-hp}{mp-mp^2-hp+h+1-p}\\
			&= 1-\frac {h+1-p}{mp-mp^2-hp+h+1-p}:=1-\alpha,
			\end{align*}
			where $\alpha=\Theta\left(\frac {h}{mp}\right)$.
			
			For the second pert, let $d\leq k\leq mp+2h$, then
		$\frac {\Pr[X=k+1]}{\Pr[X=k]} 
			\geq 1-\alpha$,
					where $\alpha=\Theta(\frac {h}{mp})$.
			Note that $\alpha\cdot h =\Omega(1)$ and thus $1-(1-\alpha)^h=\Theta(1)$. Therefore,
			\begin{align*}
			\Pr\left[X\geq d\right]
			&=\sum_{k= d}^{\infty}\Pr\left[X= k\right]
			\geq \sum_{k= d}^{mp+2h}\Pr\left[X= k\right]\\
			&\geq \sum_{k= d}^{mp+2h}(1-\alpha)^{k-d}\Pr\left[X= d\right]
			= \Pr\left[X= d\right]\sum_{k= 0}^{h}(1-\alpha)^{k}\\
			&\geq \frac {c}\alpha\cdot\Pr\left[\Bin\left(m,p\right)= d\right],
			\end{align*}
			
			for some constant $c>0$.
			Thus, $\Pr[\Bin(m,p)=d]= O(\alpha) \Pr[\Bin(m,p)\geq d]$ for $\alpha=\Theta(\frac {h}{mp})$.	
	\end{proof}

One of the most famous bounds on the tails of the binomial distribution, which we use extensively in this paper, is due to  Chernoff (see, e.g.,
\cite{AloSpe2008}, \cite{JLR}).  
%
%
 
\begin{lemma}\label{Che}
		Let $X\sim \Bin(n,p)$,	$\mu=\mathbb{E}(X)$ and $a\geq 0$, then
		\begin{enumerate}
			\item $\Pr\left[X\leq \mu-a\right]\leq \exp\left(-\frac{a^2}{2\mu}\right)$;
			\item $\Pr\left[X\geq \mu+a\right]\leq\exp\left(-\frac{a^2}{2(\mu+\frac a3)}\right)$.
		\end{enumerate}
\end{lemma}

\begin{lemma}\label{Che0}
	Let $X\sim \Bin(n,p)$.
	 Then, $\Pr[X\geq k]\leq (\frac {{\rm{e}}np}{k})^k.$
\end{lemma}

\begin{proof}
	Recall that $\binom n\ell\leq (\frac {{\rm{e}}n}\ell)^\ell$. Since $X\sim \Bin(n,p)$ we have $\Pr[X\geq k]\leq \binom nk p^k\leq (\frac {{\rm{e}}n}k)^kp^k$. 
\end{proof}
%

%

%
%

	The next lemma is due to DeMoivre-Laplace, and can be found in \cite{Bol98}.
	\begin{lemma}\label{lem:DeMoLap}
		Let $m$ be a sufficiently large integer and let $0<p:=p(m)<1$ such that $mp(1-p)\to\infty$. Let $X\sim \Bin(m,p)$ and let $d=mp+x\sqrt {mp(1-p)}$ where $x=o( \sqrt[6] {mp(1-p)})$ and $x\to \infty$. Then, $$\Pr[X\geq d]=(1+o(1)) {\frac 1{x\sqrt{2\pi}}}{\rm{e}}^{-x^2/2}={\rm{e}}^{-\frac {x^2}2-\log x-\frac 12\log (2\pi) +o(1)}.$$
	\end{lemma}
	
%

In the next two lemmas we show that the tail probabilities of two binomial random variables with close parameters are asymptotically equal. 
\begin{lemma}\label{lem:BinChange1}
	Let $m$ be a positive integer and let $p:=p(m)=o(1)$ such that $mp(1-p)\to\infty$. Let $t$ be an integer  and let $x$  be such that:
	\begin{itemize}
		\item $x=\frac {t-mp}{\sqrt{mp(1-p)}}\to \infty$,
		\item $x = o((mp(1-p))^{1/6})$,
	\end{itemize}   
	 Let $\alpha:=\alpha(m)$ and $\beta:=\beta (m)$ be such that $|\alpha|=o(1)$, $|\beta|=o(1)$, $m(1-\alpha)$ is an integer and $\frac {{(1-\alpha)(1+\beta)(1-p)}}{ {1-p(1+\beta)}}=1+o(\frac 1{x^2})$. Then, for $X\sim \Bin(m,p)$ and $X'\sim \Bin(m',p')$, 
	$$\Pr\left[X'\geq t' \right]= (1+o(1))\Pr[X\geq t],$$
	where $m'=m(1-\alpha)$, $p'=p{(1+\beta)}$, and $t'=\lceil t(1-\alpha)(1+\beta)\rceil$
\end{lemma}

\begin{proof}
	Set $a=t'-t(1-\alpha)(1+\beta)<1$ and write 
	$$x'=\frac {t'-m'p'}{\sqrt {m'p'(1-p')}}=\frac {t-mp+\frac a{(1-\alpha)(1+\beta)}}{\sqrt {mp(1-p)}}\cdot \sqrt{\frac {{(1-\alpha)(1+\beta)}}{\frac {1-p(1+\beta)}{1-p}}}=x\cdot \sqrt{\frac {{(1-\alpha)(1+\beta)(1-p)}}{ {1-p(1+\beta)}}}+o\left(\frac 1{x^2}\right).$$
	Since $x'=(1+o(1))x$, by Lemma \ref{lem:DeMoLap} we deduce,
	\begin{align*}
	\Pr\left[X'\geq t'\right]={\rm{e}}^{-\frac {x^2}2\left(\frac {{(1-\alpha)(1+\beta)(1-p)}}{ {1-p(1+\beta)}}\right)-\log x-\frac 12\log \left(\frac {{(1-\alpha)(1+\beta)(1-p)}}{ {1-p(1+\beta)}}\right)-\frac 12\log (2\pi) +o(1)}. 
	\end{align*}
	Note that $\log \left(\frac {{(1-\alpha)(1+\beta)(1-p)}}{ {1-p(1+\beta)}}\right)=o(1)$ and recall that $\frac {{(1-\alpha)(1+\beta)(1-p)}}{ {1-p(1+\beta)}}=1+o(\frac 1{x^2})$. Thus, again by Lemma~\ref{lem:DeMoLap},
	\[
	\Pr\left[X'\geq t'\right]={\rm{e}}^{-\frac {x^2}2-\log x-\frac 12\log (2\pi) +o(1)}=(1+o(1))\Pr[X\geq t]. 
\qedhere\]
\end{proof}

\begin{lemma}\label{BinEq}
	Let $n$ be an integer and $n'\leq n$ such that $n'=n(1-o(1))$. Let $0\leq p:=p(n)\leq 1$ and $p'\leq p$ such that $p'=p(1-o(1))$. Let $k\leq n'$ be an integer. Let $X\sim \Bin (n,p)$. Assume there exists an integer $s$ such that  $k\leq s\leq n'$, $\Pr[X\geq s]=o(\Pr[X\geq k])$, $\frac {n-n'}{n'}=o(\frac 1s)$, and $\frac {p-p'}p=o(\frac 1s)$. Then for  $X'\sim \Bin(n',p')$ we have $\Pr[X\geq k]=(1+o_n(1))\Pr[X'\geq k]$.
	\end{lemma}

\begin{proof}
	 The direction $\Pr[X\geq k]\geq \Pr[X'\geq k]$ follows by Observation~\ref{obs:BinDec}. For showing the other direction,  
	recall that $\Pr[X\geq s]=o(\Pr[X\geq k])$, and thus $\sum_{d=k}^{s}\binom {n}d p^d(1-p)^{n-d}=(1-o(1))\sum_{d=k}^{n}\binom {n}d p^d(1-p)^{n-d}$. By the assumptions we also have that ${\rm{e}}^{-2s(\frac{n-n'}{n'}+\frac{p-p'}p)}=(1+o(1))$. Therefore,
		\begin{align*}
		\Pr[X'\geq k]&=\sum_{d=k}^{n'}\binom {n'}d (p')^d(1-p')^{n'-d}
		\geq	\sum_{d=k}^{n'}\binom {n'}d (p')^d(1-p)^{n-d}\\
		&=	\sum_{d=k}^{n'}\binom {n}d p^d(1-p)^{n-d}\cdot \frac {(n-d)(n-d-1)\dots (n'-d+1)}{n(n-1)\dots(n'+1)}\cdot \left(\frac {p'}p \right)^d\\
		&\geq \sum_{d=k}^{n'}\binom {n}d p^d(1-p)^{n-d}\cdot \left(\frac {n'-d}{n'}\right)^{n-n'}\cdot \left(\frac {p'}p \right)^d\\
		&\geq\sum_{d=k}^{n'}\binom {n}d p^d(1-p)^{n-d}\cdot {\rm{e}}^{-2d\frac{n-n'}{n'}}\cdot {\rm{e}}^{-2d\frac{p-p'}p} 
		\geq\sum_{d=k}^{s}\binom {n}d p^d(1-p)^{n-d}\cdot {\rm{e}}^{-2d(\frac{n-n'}{n'}+\frac{p-p'}p)} \\
		&\geq {\rm{e}}^{-2s(\frac{n-n'}{n'}+\frac{p-p'}p)}\sum_{d=k}^{s}\binom {n}d p^d(1-p)^{n-d}
		= (1-o(1))\sum_{d=k}^{s}\binom {n}d p^d(1-p)^{n-d}\\
		&= (1-o(1))\sum_{d=k}^{n}\binom {n}d p^d(1-p)^{n-d}
		=(1-o(1))\Pr[X\geq k]. \qedhere
		\end{align*}
	\end{proof}

\vspace{3mm}
\subsection{Properties of random graphs and random multigraphs}\label{sec:Prop}

We start with two simple properties of random graphs.
The first claim shows that in the sparse regime, a typical random graph has no cycles. We leave the proof of this claim as an (easy) exercise.

\begin{claim}\label{GnpForest}
	Let $p=o(\frac 1n)$ then \whp $G\sim G(n,p)$ is a forest.
\end{claim}

The following theorem is a standard result regarding the maximum degree of a random graph in the relatively dense regime.

\begin{theorem}[See, e.g., Theorem 3.9 in \cite{Bol98}]\label{GnpUniMax}
	Let $p\leq \frac 12$ and $p=\omega(\frac {\log n}n)$. Then \whp $G\sim G(n,p)$ has a unique vertex of maximum degree.
\end{theorem}

Recall that in the $M(n,m)$ model each multigraph has vertex set $[n]$ and $m$ pairs from $[n]$ are chosen independently, uniformly at random. 

\begin{observation}
	For any pair $u,v\in [n]$, $\mu(uv)\sim \emph{Bin}(m,\frac 1{\binom n2})$.
\end{observation}

%

\begin{observation}
	For any vertex $v\in [n]$, $d(v)\sim \emph{Bin}(m,\frac {n-1}{\binom n2}=\frac 2n)$.
\end{observation}



%

In the following observation one can find a typical upper bound for the maximum degree of a random multigraph.

\begin{observation}\label{obs0:maxdeg}
	Let $m\geq \frac n{\log \log n}$ and $M\sim \mnm$. Then \whp $\Delta(M)\leq \frac {2m\log^2n}n$. Moreover, $\Pr\left[d(v)>\frac {2m\log^2 n}n\right]<n^{-5}$ for $v\in [n]$.
\end{observation}

\begin{proof}
	Since $d(v)\sim \Bin(m,\frac 2n)$ we have by Lemma \ref{Che0},
	$$\Pr\left[d(v)>\frac {2m\log^2 n}n\right]\leq {\rm{e}}^{-\frac {2\log^2n}{\log\log n}}<\frac 1{n^5}.$$ By the union bound the claim follows.
\end{proof}


The next observation refers to the  edge-multiplicities of a random multigraph for different values of $m$. 

\begin{observation}\label{obs:multiplicity}
	Let $n$ ne an integer and $m:=m(n)$. Let $M\sim \mnm$. Then the following hold.
	\begin{enumerate}[$(1)$]
		\item If $m\leq n\log^{100} n$, then 
		the probability that there exists an edge $e$ with $\mu(e)\geq 3$ is at most $O(\frac {\log^{300}n}n)$.
		\item If $m< n^2\log^2 n$, then with probability $1-o(\frac 1{n^4})$, for every $e\in \binom {[n]}2$, one has $\mu(e)\leq 30\log^2 n$.
		\item If $m=\omega (n^2\log n)$, then with probability $1-o(\frac 1{n^4})$, for every $e\in \binom {[n]}2$, one has $\mu(e)\in \frac m{\binom n2}\pm 4\sqrt {\frac m{\binom n2}\log n}$.
	\end{enumerate}
\end{observation}

\begin{proof}
	First 	Recall that $\mu(e)\sim \Bin(m,\binom n2 ^{-1})$. Assume that  $m\leq n\log^{100} n$.
		By Lemma~\ref{Che0}, 
		$$\Pr[\mu(e)\geq 3]
		=O\left(\frac {\log^{300}n}{n^3}\right).$$
		Then by the union bound, the probability that there exists an edge $e$ with $\mu(e)\geq 3$ is at most $O(\frac {\log^{300}n}n)$.
		
		Next, assume that $m< n^2\log^2 n$. 
	 Then $\mathbb E(\mu(e))<3\log^2n$, and by Lemma \ref{Che0} we have
		$$\Pr\left[\mu (e)>30\log^2n\right]\leq {\rm{e}}^{-30\log^2n}=o\left(\frac 1{n^6}\right),$$
		and by the union bound Item $(2)$ follows.
		
		Finally, assume that $m=\omega (n^2\log n)$. Take $a=4\sqrt {\frac m{\binom n2}\log n}$, by Lemma \ref{Che},
		$$\Pr\left[\mu(e)\geq \mathbb E(\mu(e))+a\right]\leq {\rm{e}}^{-\frac {a^2}{2(\mathbb E(X)+a/3)}}={\rm{e}}^{-\frac {16\log n\mathbb E(X)}{2(\mathbb E(X)+o(\mathbb E(X)))}}\leq {\rm{e}}^{-7\log n}=o\left(\frac 1{n^6}\right).$$
		By the union bound, with probability $1-o(\frac 1{n^4})$, for every $e\in \binom {[n]}2$ we have $\mu(e)\leq \mathbb E(\mu(e))+a=\frac m{\binom n2}+4\sqrt {\frac m{\binom n2}\log n}$.
		
		Similarly, by Lemma \ref{Che},
		$$\Pr\left[\mu(e)\leq \mathbb E(\mu(e))-a\right]\leq {\rm{e}}^{-\frac {a^2}{2\mathbb E(X)}}={\rm{e}}^{-\frac {16\log n\mathbb E(X)}{2\mathbb E(X)}}\leq {\rm{e}}^{-7\log n}=o\left(\frac 1{n^6}\right).$$
		By the union bound, with probability $1-o(\frac 1{n^4})$, for every $e\in \binom {[n]}2$ we have $\mu(e)\geq \mathbb E(\mu(e))-a=\frac m{\binom n2}-4\sqrt {\frac m{\binom n2}\log n}$. 
\end{proof}

%
%
%
%
\begin{corollary}\label{cor:DifMu}
	Let $m=\omega (n^2\log n)$ and $M\sim \mnm$ and let $\mu(M)=\max\{\mu(e) \mid e\in \binom {[n]}2 \}$, $\mu_{min}=\min\{\mu(e) \mid e\in \binom {[n]}2 \}$. Then \whp $\mu(M)-\mu_{min}=O\left(\sqrt {\frac{m\log n}{n^2}}\right)$.
\end{corollary}

We will use the following claims in the proofs of Theorem \ref{GSforRM} and Theorem \ref{thm:nEven}. The proofs of the claims are rather technicals and can be found in the next subsection. The first claim shows that the typical maximum degree of a  random multigraph is not too far from its expectation. 	Let $\varepsilon\in (0,1)$ be a constant and define $d_+=\frac {2m}n+(1+ \varepsilon)\sqrt {\frac {4m}n(1-\frac 2n)\log n}$. 

\begin{claim}\label{lem0:MaxDeg2}
Let $n$ be a sufficiently large integer, and let $m=\omega (n\log^3n)$. Let $M\sim \mnm$.  Then \whp  $\Delta(M)\leq d_+$.  
\end{claim}
%

In the next lemma we show that the typical gap between the two largest degrees in a random multigraph is quite large. This phenomena is already known in the context of random graph (see, e.g., \cite{Bol98}), and here we prove an analogous statement for random multigraphs.

\begin{lemma}\label{lem0:DegDiff}
	Let $n$ be sufficiently large integer and let $m=\omega (n\log^5n)$. Let $M\sim \mnm$, and let $d_1 \geq d_2 \geq \dots \geq d_n$ be the (ordered) degree sequence of $M$.  Then \whp $d_1-d_2\geq\frac {1}{f(n)}\sqrt{\frac {2m}{n \log n}}$, where $f(n)=o(\log n)$ is a function tending to infinity arbitrarily slowly with $n$. 
\end{lemma}

\vspace{3mm}
\subsubsection{Proofs of Claim \ref{lem0:MaxDeg2} and Lemma \ref{lem0:DegDiff}}

We will use the following notation. Let $X_d$ be the number of vertices of degree $d$ in $M\sim \mnm$, and let $Y_d=\sum_{\ell\geq d}X_\ell$. Note that $\mathbb E(Y_d)=\sum_{v\in [n]}\Pr[d(v)\geq d]$.

In both proofs we will use the following approximation.
\begin{remark}\label{re:Y_s}
	 	Let $\alpha>-1$ be a constant and $Y_s=\sum_{d\geq  d_s}X_d$ where $d_s=\left\lceil mp+(1+\alpha)\sqrt {2mp(1-p)\log n}\right\rceil$. Note that here $p=\frac 2n$ and $x:=\left(1+\alpha+o\left(\frac 1{\sqrt {mp}}\right)\right)\sqrt{2\log n}=o(\sqrt[6] {m/n})$. Then
	By Lemma~\ref{lem:DeMoLap},
	\begin{align*}
	\mathbb E(Y_s) 
	&= n(1+o(1))\frac 1{x\sqrt {2\pi}}{\rm{e}}^{-x^2/2}
	= (1+o(1))\frac n{(1+\alpha)\sqrt {4\pi\log n}}{\rm{e}}^{-(1+\alpha)^2\log n}
	=\Theta\left( {\frac {n^{-2\alpha-\alpha^2}}{\sqrt {\log n} }}\right).
	\end{align*}
\end{remark}

We first prove Claim \ref{lem0:MaxDeg2}.
\begin{proof}[Proof of Claim \ref{lem0:MaxDeg2}]
		Let $Y_+=\sum_{d\geq d_+}X_d$. By Remark~\ref{re:Y_s} we have 
	\begin{align*}
	\mathbb E(Y_+) 
	&= \Theta\left( {\frac {n^{-2\varepsilon -\varepsilon^2}}{\sqrt {\log n} }}\right)=o\left(\frac 1{n^{\varepsilon}}\right).
	\end{align*}
	By Markov's inequality the claim follows.
\end{proof}

\medskip

The proof of Lemma \ref{lem0:DegDiff} is  more complicated, and we need the following statements.

Let $X\sim \Bin(m,\frac 2n)$ and let $\varepsilon\in (0,1)$. Let $d_-$ be the maximal integer such that $\Pr[X\geq d_-]\geq \frac{f^{1/3}(n)}n$ where $f(n)$ is as in Lemma~\ref{lem0:DegDiff}. Let $Y_-=\sum_{d\geq d_-}X_d$.

\begin{remark}\label{re:d-}
The following hold.
	\begin{enumerate}[$(1)$]
		\item $\mathbb E (Y_-)\geq f^{1/3}(n)$.
		\item If $d_\varepsilon=\left\lceil \frac {2m}n+(1-\varepsilon)\sqrt{\frac {4m}n (1-\frac 2n)\log n}\right\rceil$ and $Y_\varepsilon= \sum_{d\geq d_\varepsilon}X_d$, then by Remark~\ref{re:Y_s}, we have that $\mathbb E(Y_\varepsilon)\geq n^{\varepsilon}$, and therefore $d_->\frac {2m}n+(1-\varepsilon)\sqrt{\frac {4m}n (1-\frac 2n)\log n}$.
		\item If $d_{up}=\left\lceil\frac {2m}n+(1+\frac \varepsilon2)\sqrt{\frac {4m}n (1-\frac 2n)\log n}\right\rceil$ and $Y_{up}=\sum_{d\geq d_{up}}X_d$, then by Remark~\ref{re:Y_s}, we have that $\mathbb E(Y_{up})=o(1)$. Combining it with  item $(1)$ and the fact that the tail of the binomial random variable decays slowly (see Claim~\ref{claim:alphaBin}) we have $d_-<\frac {2m}n+(1+\frac\varepsilon2)\sqrt{\frac {4m}n (1-\frac 2n)\log n}$.
		\item From items $(2)$,$(3)$, we have that $d_+-d_- \in [\frac \varepsilon2\sqrt{\frac {4m\log n}n}, 2\varepsilon\sqrt{\frac {4m\log n}n}]$. 
	\end{enumerate}
\end{remark}

\begin{claim}\label{claim:sBin}
	Let $ m =\omega (n\log^5 n)$. Let $X\sim \Bin \left(m,\frac {2}{n}\right)$ and $Y\sim \Bin \left(m-d_-,\frac {2}{n-1}\right)$.
	Assume that $\mu\leq \max\left\{30\log^2n, \frac m{\binom n2}+ 4\sqrt {\frac m{\binom n2}\log n}\right\}$ is a positive integer. Then 
	$$\Pr\left[Y\geq d_--\mu\right]\leq (1+o(1))\Pr\left[X\geq d_-\right].$$ 
\end{claim}

\begin{proof}
	Let $\alpha=\frac {d_-}m$, $\beta=\frac 1{n-1}$. Write $d_-=\frac {2m}n+x\sqrt{\frac {2m}n (1-\frac 2n)}$, then by Remark~\ref{re:d-}, $x=\Theta (\sqrt{\log n})$. Let $d'=\lceil d_-(1-\alpha)(1+\beta)\rceil$ By Lemma \ref{lem:BinChange1} we have that 
	$$\Pr\left[Y\geq d'\right]= (1+o(1))\Pr\left[X\geq d_-\right].$$
	Now, if $d_--\mu\geq d'$ then we are done. If not, then recall that by Remark~\ref{re:d-}, there exist  constants $c_1,c_2>0$ such that $\frac {2m}n + c_1 \sqrt{\frac {2m\log n}n}\leq d_-\leq\frac {2m}n + c_2 \sqrt{\frac {2m\log n}n}$.
	Now, let $\ell>0$ be such that $d_--\mu+\ell=d'$, and note that $\ell$ is an integer and that $(1-\alpha)(1+\beta)=1-\frac {d_-n}{m(n-1)}+\frac 1{n-1}$. By the condition on $\mu$ we have that $\mu\leq 30\log^2n+ \frac m{\binom n2}+ 4\sqrt {\frac m{\binom n2}\log n}$. Then
	\begin{align*}
	\ell &
	\leq \mu+d_-(1-\alpha)(1+\beta)+1-d_- = \mu+d_-\left(-\frac {d_-n}{m(n-1)}+\frac 1{n-1}\right)+1
	=\mu-\frac {d_-^2n}{m(n-1)}+\frac{d_-}{n-1} +1\\
	&\leq  \mu - \frac n{m(n-1)}\left(\frac {2m}n + c_1 \sqrt{\frac {2m\log n}n}\right)^2+\frac 1{n-1}\left(\frac {2m}n + c_2 \sqrt{\frac {2m\log n}n}\right)+1\\
	& = \mu  -\frac {2m}{n(n-1)} -  {\frac {2c^2_1\log n}{n-1}} + (c_2-4c_1) \sqrt{\frac {2m\log n}{n(n-1)^2}}+1\\
	& \leq 30\log^2n+ \frac {2m}{n(n-1)}+4\sqrt {\frac {2m\log n}{n(n-1)}}  -\frac {2m}{n(n-1)} -  {\frac {2c^2_1\log n}{n-1}} + (c_2-4c_1) \sqrt{\frac {2m\log n}{n(n-1)^2}}+1\\
	& = O\left( \sqrt {\frac {m\log n}{n(n-1)}} +\log^2n \right) .
	\end{align*}
	
	Write $m'=m-d_-$, $p'=\frac 2{n-1}$.
	Set $h=d_--\mu-m'p'=d'-\ell-m'p'$, then $h>0$ and in fact $h=\Theta\left(\sqrt {\frac {m\log n}{n}}\right)$.  
%
	 Therefore by Claim \ref{claim:alphaBin} we have that $\Pr[Y=d']>(1-\alpha)^\ell\Pr[Y=d'-\ell]$ for $\alpha=O(\frac {\sqrt{m\log n}}{\sqrt n m'p'})$. Since $\alpha\ell=o(1)$, we have $(1-\alpha)^\ell=(1-o(1))$. Thus,
	\begin{align*}
	\Pr\left(Y\geq d'\right) 
	&=\sum_{k= d'}^{\infty}\Pr\left(Y= k\right)
	\geq \sum_{k= d'}^{\infty}(1-\alpha)^{\ell}\Pr\left(Y= k-\ell\right)
	\geq (1-o(1))\sum_{k= d'}^{\infty}\Pr\left(Y= k-\ell\right)\\
	&=(1-o(1)) \Pr\left(Y\geq d'-\ell\right)
	=(1-o(1)) \Pr\left(Y\geq d_--\mu\right).
	\end{align*}
	 
	 All in all, we have that 
	 $\Pr\left(Y\geq d_--\mu\right)\leq (1+o(1))\Pr\left(X\geq d_-\right).$
\end{proof}

\begin{remark}\label{re:condi}
	Let $k_1,k_2$ be integers. Let $M\sim M(n,m)$  and assume that $u,v\in V(M)$ are two vertices. Recall that $d(u)\sim \Bin(m,\frac 2n)$ and that $\mu(uv)\leq \mu(M)$. Also, conditioning on the event $d(v)\geq k_2$, the probability that a new edge is in $E[u,V\setminus \{u,v\}]$ is $\frac 2{n-1}$ (that is, the probability for an edge to contain $u$ in the sub(multi)graph $M-v$). Let $d'(u)\sim \Bin(m-k_2,\frac 2{n-1})$.  Then by the monotonicity of the binomial distribution (see Observation~\ref{obs:BinDec}) we have that  
	$$\Pr\left[d(u)\geq k_1 \mid d(v)\geq k_2 \right]\leq \Pr[d'(u)\geq k_1-\mu(M)].$$
	Similarly,
	$\Pr\left[d(u)= k_1 \mid d(v)= k_2 \right]\leq \Pr[d'(u)= k_1-\mu(M)].$
\end{remark} 

\begin{claim}\label{claim:d_-}
	Let $m=\omega (n\log^5 n)$, and
	let $Y_-=\sum_{d\geq d_-}X_d$. Then \whp $Y_-\geq 2$.
\end{claim}

\begin{proof}
	Recall that $Y_-=\sum_{v\in [n]}\textbf{1}_{d(v)\geq d_-}$, therefore  $\mathbb E(Y_-)=\sum_{v\in [n]}\Pr(d(v)\geq d_-)$ and $\mathbb E(Y_-^2)=\sum_{v,u\in [n]}\Pr(d(v)\geq d_-,\ d(u)\geq d_-)$. Moreover, by Observation  \ref{obs:multiplicity} we have that\\ $\mu(M)\leq \max \{30\log^2n, \frac m{\binom n2}+ 4\sqrt {\frac m{\binom n2}\log n}\}$ with probability $1-o(\frac 1{n^4})$. 
	Let $X\sim \Bin \left(m,\frac {2}{n}\right)$ and $Y\sim \Bin \left(m-d_-,\frac {2}{n-1}\right)$.	Thus, by Claim \ref{claim:sBin} and Remark~\ref{re:condi}

	\begin{align*}
	\mathbb E(Y_-(Y_--1))&=\mathbb E(Y_-^2)-\mathbb E(Y_-)=\sum_{v\neq u\in [n]}\Pr[d(v)\geq d_-,\ d(u)\geq d_-]\\
	&=n(n-1)\Pr[d(v_1)\geq d_-,\ d(v_2)\geq d_-]\\
	&=n(n-1)\Pr[d(v_1)\geq d_-]\Pr[d(v_2)\geq d_- \mid d(v_1)\geq d_-]\\
	&\leq n(n-1)\Pr\left[X\geq d_-\right]\left(\Pr\left[Y\geq d_--\mu(M)\right]+o\left(\frac 1{n^4}\right)\right)\\
	&\leq (1+o(1))n^2\Pr\left[X\geq d_-\right]\left(\Pr\left[X\geq d_- \right]+o\left(\frac 1{n^4}\right) \right)\\
	&=(1+o(1))(\mathbb E(Y_-))^2.
	\end{align*}
	
	Therefore, $\Pr[Y_{-}<\frac 12 \mathbb E(Y_{-})]<\frac {\mathbb E(Y_{-}(Y_{-}-1))+\mathbb E(Y_{-})-(\mathbb E(Y_{-}))^2}{(\mathbb E(Y_{-})^2)/4}\leq o(1)+\frac {4}{\mathbb E(Y_{-})}=o(1)+O(\frac 1{f^{1/3}(n)})=o(1)$. Thus \whp $Y_-\geq 2$.
	%
	%
\end{proof}

\medskip

We are now ready to prove Lemma \ref{lem0:DegDiff}.
\begin{proof}[Proof of Lemma \ref{lem0:DegDiff}]
%
%
	We will actually prove something stronger.  We will prove that \whp for every $u,v\in \binom {[n]}2$ such that $d(v),d(u)\in [d_-,d_+]$ we have $|d(v)-d(u)|\geq \frac {1}{f(n)}\sqrt{\frac {2m}{n \log n}}$. Since by Claim~\ref{lem0:MaxDeg2} and Claim~\ref{claim:d_-}  we have that \whp $d_1,d_2\in [d_-,d_+]$, the statement will follow.
	
	Let $X\sim \Bin\left(m,\frac 2{n}\right)$ and $Y\sim \Bin\left(m-d_-,\frac 2{n-1}\right)$.
	By Claim \ref{claim:sBin} we have $\Pr\left[Y\geq d_--\mu\right]\leq(1+o(1))\Pr\left[X\geq d_-\right]$ for $\mu\leq \max \left\{30\log^2n,\frac m{\binom n2}+ 4\sqrt {\frac m{\binom n2}\log n}\right\}$.  
	Let $\ell_1,\ell_2\in [d_-,d_+]$, $m'=m-\ell_1$, and $p'=\frac 2{n-1}$. Let $h=\ell_2-\mu(M)-m'p'=\Theta\left(\sqrt{\frac {m\log n}n}\right)$ and $Y_{\ell_1}\sim \Bin\left(m-\ell_1,\frac {2}{n-1}\right)$. 	
	Thus, by Claim~\ref{claim:alphaBin} we have $\Pr\left[Y_{\ell_1}= \ell_2-\mu(M)\right]\leq \alpha\Pr\left[Y_{\ell_1}\geq \ell_2-\mu(M)\right]$, for  $\alpha=\Theta (\frac {nh}{m})=O \left(\sqrt{\frac {n\log n}m}\right)$.
	Also, note that $\Pr\left[ X\geq d_-\right]\leq (2+o(1))\Pr\left[ X\geq d_-+1\right]$. Indeed, by Claim~\ref{claim:alphaBin}
	\begin{align*}
	\Pr\left[ X\geq d_-\right]&=\Pr\left[ X= d_-\right]+\Pr\left[ X\geq d_-+1\right]\\
	&\leq \frac 1{1-\alpha}\Pr\left[ X= d_-+1\right]+\Pr\left[ X\geq d_-+1\right]\\
	&\leq \left(1+\frac 1{1-\alpha}\right)\Pr\left[ X\geq d_-+1\right]=(2+o(1))\Pr\left[ X\geq d_-+1\right].
	\end{align*} 
	Since by Observation \ref{obs:multiplicity} we have that $\mu(M)\leq \max \left\{30\log^2n,\frac m{\binom n2}+ 4\sqrt {\frac m{\binom n2}\log n}\right\}$ with probability $1-o(\frac 1{n^4})$, and by Lemma \ref{lem0:MaxDeg2} and Claim~\ref{claim:d_-}  we have that \whp $d_1,d_2\in [d_-,d_+]$, 
	using Remark~\ref{re:condi} and Observation~\ref{obs:multiplicity} we conclude that
	
	\begin{align*}
	&\Pr[\exists u,v\in [n],\ d(u),d(v)\in [d_-,d_+],\ |d(u)-d(v)|\leq t]
	\leq n^2\Pr[d(u),d(v)\in [d_-,d_+],\  |d(u)-d(v)|\leq t]\\
	&\leq (1+o(1))n^2\sum_{\ell_1=d_-}^{d_+}\sum_{\substack{\ell_2=d_-\\  |\ell_1-\ell_2|\leq t}}^{d_+}\Pr[d(u)=\ell_1,\ d(v)=\ell_2]\\
	&\leq (1+o(1)) n^2\sum_{\ell_1=d_-}^{d_+}\sum_{\substack{\ell_2=d_-\\  |\ell_1-\ell_2|\leq t}}^{d_+}\Pr[d(u)=\ell_1]\Pr[ d(v)=\ell_2 \mid d(u)=\ell_1]\\
	&\leq (1+o(1)) n^2\sum_{\ell_1=d_-}^{d_+}\sum_{\substack{\ell_2=d_-\\  |\ell_1-\ell_2|\leq t}}^{d_+}\Pr[d(u)=\ell_1]\Pr\left[ Y_{\ell_1}=\ell_2-\mu(M)\right]\\
	&\leq (1+o(1)) n^2\sum_{\ell_1=d_-}^{d_+}\sum_{\substack{\ell_2=d_-\\  |\ell_1-\ell_2|\leq t}}^{d_+}\Pr[d(u)=\ell_1]\alpha\Pr\left[Y_{\ell_1}\geq \ell_2-\mu(M)\right]\\
	&\leq (1+o(1))n^2\alpha\Pr\left[Y\geq d_--\mu(M)\right]\sum_{\ell_1=d_-}^{d_+}\sum_{\substack{\ell_2=d_-\\  |\ell_1-\ell_2|\leq t}}^{d_+}\Pr[d(u)=\ell_1]\\
	&\leq (1+o(1))\cdot n^2(2t+1)\alpha\left(\Pr\left[ X\geq d_-\right]+o\left(\frac 1{n^4}\right)\right)\sum_{\ell_1=d_-}^{d_+}\Pr[d(u)=\ell_1]\\
	&\leq (1+o(1))\cdot n^2(2t+1)\alpha\left(\Pr\left[ X\geq d_-\right]+o\left(\frac 1{n^4}\right)\right)\Pr\left[ X\geq d_-\right]\\
	&\leq (1+o(1))\cdot n^2(2t+1)\alpha \Pr\left[ X\geq d_-\right]^2
	\leq (2+o(1))\cdot n^2(2t+1)\alpha\Pr\left[ X\geq d_-+1\right]^2\\
	&=O \left(n^2t\alpha \left(\frac {f^{1/3}(n)} {n}\right)^2\right)
	= O (t\alpha f^{2/3}(n)).
	\end{align*}

	All in all, for $t=\frac {1}{f(n)}\sqrt{\frac {2m}{n \log n}}$ and $\alpha=O\left(\sqrt{\frac {n\log n}m}\right)$, we have
	\[\Pr[d_1-d_2<t]= O (t\alpha f^{2/3}(n))=o(1). \qedhere\]
\end{proof}

\vspace{5mm}
\section{The chromatic index of random graphs}
In this section we will prove Theorem \ref{thm:Gnp}. 

The proof of Theorem \ref{thm:Gnp} is a simpler version of the proof of Theorem \ref{GSforRM} for small values of $m$. For a warm-up, we included the details below.

Let $G\sim G(n,p)$.
First we will handle the case where $p=\omega(\frac {\log n}n)$ and $p\leq \frac 12$. By Theorem~\ref{GnpUniMax}, \whp there exists a unique vertex of maximum degree. Then, by Theorem~\ref{vizing} we have that \whp $\chi'(G)=\Delta(G)$.

Next, let us consider the sparse case, where $p=o(\frac 1n)$. By Lemma \ref{GnpForest}, $G\sim G(n,p)$ is \whp a forest and therefore by Theorem \ref{vizing} \whp $\chi'(G)=\Delta(G)$.

Thus, we may assume from now on that $\frac 1{n\log\log n}\leq p\leq \frac {\log^2n}n$. 
Let $G\sim G(n,p)$, recall that $d(v)\sim \Bin(n-1,p)$ for $v\in [n]$. We will consider the vertices of \textit{very large} degree.  Let $X\sim \Bin(n-1,p)$, let $d_0=\max \{d \mid \Pr[X\geq d]\geq n^{-0.9}\ \}$, and define $L_0=\{v\in [n]\mid d(v)\geq d_0 \}$. Note that for $v\in [n]$, $\Pr[d(v)\geq d_0+1]<n^{-0.9}$ and $\Pr[v\in L_0]\geq n^{-0.9}$. 

If we can show that \whp the set $L_0$ is an independent set and not empty, then clearly \whp the set of vertices of maximum degree forms an independent set, and then by Theorem \ref{vizing} we are done. For proving that \whp every two vertices in $L_0$ are not neighbours we will use the following two claims. 

\begin{observation}\label{obs:Bin}
	Since $d(v)\sim \Bin(n-1,p)$ and $d_0< \log^3n$, we have by Observation~\ref{obs:BinFrac} that $\Pr[d(v)\geq d_0]\leq (1+\frac{4d_0}{(n-1)p})\Pr[d(v)\geq d_0+1]<(4\log^4n+1)n^{-0.9}$, and similarly $\Pr[d(v)\geq d_0-1]<(1+\frac{4d_0}{(n-1)p})\Pr[d(v)\geq d_0]
	\leq n^{-0.9}\log^{8}n$.
\end{observation}

First, we will bound from above the probability that a vertex is in $L_0$.

\begin{claim} \label{claim:PrL0Gnp}
	Let $G\sim G(n,p)$, $\frac 1{n\log\log n}\leq p=o(1)$. Then for $v\in [n]$,  $\Pr[v\in L_0]= O(n^{-0.9}\log^4 n)$.
\end{claim}

\begin{proof}

	 By Lemma \ref{Che0}, 
	$\Pr[d(v)\geq \log^2n\cdot \mathbb{E}[d(v)]]\leq {\rm{e}}^{-\log^2n \mathbb{E}[d(v)]}<n^{-0.9},$
	where the last inequality comes from the fact that $\mathbb E{[d(v)]}>\frac 1{2\log\log n}$.
	That is, $d_0<\log^2n\mathbb{E}[d(v)]\leq \log^2n\cdot np$. 
%
	By Observation~\ref{obs:Bin}	we have 
	$\Pr[v\in L_0]<(4\log^4n+1)n^{-0.9}. \qedhere $
\end{proof}

\medskip
Next, we will show that the set $L_0$ is not empty.

\begin{claim}\label{claim:L0nonempty_G}	Let $G\sim G(n,p)$. Then \whp there exists $v\in L_0$.
\end{claim}

\begin{proof}
First, let $n'=n-n^{0.95}$, $s=\log^{4}n$, and $k=d_0$. Let $X\sim \Bin(n-1,p)$ and $X'\sim \Bin(n-n^{0.95},p)$. Note that $np\leq \log^3n$ and thus by Lemma~\ref{Che0} we have that $\Pr[X\geq \log^4n]\leq {\rm{e}}^{-\log^4n}=o(n^{-0.9})$. Thus $d_0<\log^4n$. Also, $\frac {n-n'}{n'}=\Theta(\frac 1{n^{0.05}})$, thus by Lemma \ref{BinEq}, we have that $\Pr[X\geq d_0]=(1+o(1))\Pr[X'\geq d_0]$. 

Now, choose a subset $U\subset [n]$ of size $n^{0.95}$. Let $X_U$ be the number of vertices $v$ in $U$ for which $e(\{v\},[n]\setminus U)\geq d_0$.
	Then $X_U\sim \Bin(|U|,\tilde{p})$ where $\tilde{p}=\Pr[e(\{v\},[n]\setminus U)\geq d_0]$ for some $v\in U$. Note that $\tilde{p}=\Pr[e(v,[n]\setminus U)\geq d_0]=\Pr[X'\geq d_0]\geq (1-o(1))\Pr[X\geq d_0]\geq n^{-0.9}$. Therefore, $\mathbb E(X_U)=n^{0.95}\tilde{p}\geq n^{0.05}$. By Lemma \ref{Che}, $\Pr(X_U\leq \frac 12n^{0.05})\leq {\rm{e}}^{-\frac 18n^{0.05}}=o(1)$. Thus \whp $X_U\geq 1$ and $|L_0|\geq 1$.
\end{proof}

\medskip

We are now ready to show that \whp $L_0$ is an independent set.

\begin{claim}\label{claim:L0FarGnp}
	Let $G\sim G(n,p)$. Then for every $v,u\in L_0$, \whp $u\nsim v$.
\end{claim}

\begin{proof}
	Let $X_i\sim \Bin(n-i,p)$, and
	let $v_1,v_2$ be vertices and let $A$ be the event that $v_1\sim v_2$. Then,
	$\Pr[A]= p \leq \frac {\log^2n}n$. 
	Observe that by Claim \ref{claim:PrL0Gnp} (and Observations~\ref{obs:BinDec} and ~\ref{obs:Bin}) 
	\begin{align*}
	\Pr[v_1,v_2\in L_0 \mid A]&=\Pr[d(v_1)\geq d_0 \mid A]\Pr[d(v_2)\geq d_0 \mid A,\ v_1\in L_0]\\
	&\leq \Pr[X_2\geq d_0-1]\Pr[X_2\geq d_0-1]\\
	&\leq \Pr[X_1\geq d_0-1]^2=O(n^{-1.8}\log^{16}n).
	\end{align*}
	Going over all choices of $v_1,v_2$, the probability that there exists an edge between two vertices in $L_0$ is at most $n^{2}\Pr[A]\Pr\left[v_1,v_2\in L_0 \mid A\right]=O\left(n^2\cdot\frac {\log^2n}n \cdot n^{-1.8}\log^{16} n\right)=o(1). \qedhere $
\end{proof}

\vspace{5mm}
\section{The chromatic index of random multigraphs $\mnm$}

In this section we prove Theorems \ref{GSforRM} and \ref{thm:nEven}. We divide the proof into three cases. In Section~\ref{sec:m<nlogn} we prove that if $m\leq n\log^{100}n$ then \whp $\chi'(M)=\Delta$. In Section \ref{sec:m<n2logn} we show that that if $n\log^{100}n\leq m\leq  n^2\log^2 n$ then \whp also here $\chi'(M)=\Delta$. In Section \ref{sec:nEven} we show that if $n$ is even, and $m=\omega(n^2 \log n)$ then \whp $\chi'(M)=\Delta$, completing the proof of Theorem \ref{thm:nEven} and the proof of Theorem \ref{GSforRM} for the even case. In Section \ref{sec:OddN} we show that if $n$ is odd and $m=\omega(n^2\log n)$ then \whp $\chi'(M)=\Delta$ or $\lceil \rho(M) \rceil$. Then, in Section~\ref{sec:alg} we discuss the algorithmic aspects of our proofs.

In the entire section we let  $d_1 \geq d_2 \geq \dots \geq d_n$ be the (ordered) degree sequence of a multigraph $M$.

\vspace{3mm}
\subsection{The case $m\leq n\log^{100} n$}\label{sec:m<nlogn}

For this case we act as follows.
We first show that for very small values of $m$, say $m=O(\frac n{\log\log n})$, a typical multigraph is actually a forest. For the remaining values of $m$, we will show that the typical edge multiplicity in this case is at most two. Then, the main idea is to remove a matching from the multigraph to create a simple graph with maximum degree $\Delta-1$ for which the vertices of degree $\Delta-1$ form an independent set (similarly to the proof of Theorem~\ref{thm:Gnp}). Thus we can colour this graph with $\Delta-1$ colours, and colour the removed matching with one additional colour. \\

We start with a claim estimating the probability of a given set of edges to have large multiplicities  in $M\sim M(n,m)$.
	
	\begin{claim}\label{claim:edgeSet}
		Let $M\sim M(n,m)$ and let $r,k_1,\dots,k_r$ be  integers. Let $A=\{e_1,\dots,e_r\}$ be a set of $r$ pairs from $\binom {[n]}2$. Let $X$ be the event that  $\mu(e_i)\geq k_i$ for all $i\in [r]$. Then $\Pr[X]\leq\left(\frac {5{\rm{e}}m}{n^2}\right)^{\sum_{i=1}^r k_i}$.
	\end{claim}
	
	\begin{proof}
		Let $A_0\subseteq A$. Since $\mu(e)\sim \Bin(m,\frac 2{n(n-1)})$, we have $\Pr[\mu(e_i)\geq k_i\mid \mu(e_j)\geq k_j\ \forall e_j\in A_0 ]\leq \Pr[\mu(e_i)\geq k_i]$ for every $i$ with $e_i\notin A_0$. The inequality follows, for example, from Theorem 10 in \cite{Dub} or from the main theorem in \cite{Mc}.  By Lemma~\ref{Che0}, $\Pr[\mu(e)\geq k]\leq \left(\frac {4{\rm{e}}m}{n(n-1)}\right)^k \leq \left(\frac {5{\rm{e}}m}{n^2}\right)^k $. Then,
		\[ \Pr[X]\leq \prod_{i=1}^r\Pr[\mu(e_i)\geq k_i]\leq \left( \frac {5{\rm{e}}m}{n^2}\right)^{\sum_{i=1}^r k_i}. \qedhere\]
	\end{proof}

Let us first deal with the case that $m\leq \frac {100n}{\log\log n}$. Let $r\in [n]\setminus \{1\}$ and let $v_1,\dots,v_r$ be $r$ vertices. Denote by $X_C$ the event that the cycle $C=(v_1,\dots,v_r,v_1)$ appears in $M\sim\mnm$. Then by Claim~\ref{claim:edgeSet}, $\Pr[X_C]\leq (\Pr[\mu(v_1v_2)\geq 1])^r\leq (\frac {5{\rm{e}}m}{n^2})^r\leq(\frac {1500}{n\log\log n})^r$. Thus the probability that there exists a cycle in $M\sim \mnm$ is at most
$$\sum_{r=2}^nn^r\Pr[X_C]\leq\sum_{r=2}^n\left(\frac {1500}{\log\log n}\right)^r
=o(1).$$
Therefore, for  $m\leq \frac {100n}{\log\log n}$, $M\sim \mnm$ is \whp a forest and thus by Theorem \ref{vizing} \whp $\chi'(G)=\Delta(G)$.\\  

\noindent
From now on we will assume that $m\geq \frac n{\log\log n}$.

	Similarly to the proof of Theorem \ref{thm:Gnp}, we will consider vertices of very large degrees.
	 Let $X\sim \Bin(m,\frac 2n)$, and define 
\begin{align*}
	d_0&=\max \{d \mid \Pr[X\geq d]\geq n^{-0.9}\ \},\ \ L_0=\{v\in V\mid d(v)\geq d_0 \},\\
	d'_0&=\max \{d \mid \Pr[X\geq d]\geq n^{-0.95}\ \},\ \ L'_0=\{v\in V\mid d(v)\geq d'_0 \}.
\end{align*}
	  Note that for $v\in V$ we have
	  \begin{align*}
	  &\Pr[d(v)\geq d_0+1]<n^{-0.9},\ \ \Pr[v\in L_0]\geq n^{-0.9},\\
	  &\Pr[d(v)\geq d'_0+1]<n^{-0.95},\ \ \Pr[v\in L'_0]\geq n^{-0.95}.
	  \end{align*}
	   Furthermore, by Observation~\ref{obs0:maxdeg}, $d_0\leq d'_0<\frac {2m\log^2n}n$.\\


We would like to bound from above the probability that a vertex is in $L_0$.

\begin{claim} \label{claim:PrL0}
	Let $M\sim \mnm$. Then for $v\in [n]$,  $\Pr[v\in L_0]= O(n^{-0.9}\log^2 n)$ and $\Pr[v\in L'_0]= O(n^{-0.95}\log^2 n)$.
\end{claim}

\begin{proof}
%
%

	Since $d(v)\sim \Bin(m,\frac 2n)$ and $d_0\leq d'_0<\frac {2m\log^2n}n$,  by Observation~\ref{obs:BinFrac} we have that $$\Pr[v\in L_0]=\Pr[d(v)\geq d_0]\leq \left(1+\frac{4nd_0}{2m}\right)\Pr[d(v)\geq d_0+1]<(4\log^2n+1)n^{-0.9},$$ and similarly $\Pr[v\in L'_0]=\Pr[d(v)\geq d'_0]
	=O(n^{-0.95}\log^{2}n)$.
%
\end{proof}

\medskip

		Next, we will prove that \whp the vertices in the set $L_0$ are ``far" from each other.
		
		\begin{claim}\label{claim:L0Far}
			\whp for every $v,u\in L_0$,  there is no path of length at most 2 connecting $u$ and $v$.
		\end{claim}
		
		\begin{proof}
			Let $M\sim M(n,m)$ and let $W$ be a set of ordered vertices $\{v_1,\dots,v_r\}$,  $2\leq r\leq 3$. Let $X_P$ be the event that $P=(v_1,\dots,v_r)$ is a path in $M$. Then by Claim~\ref{claim:edgeSet}
			$\Pr[X_P]= O((\frac m{n^2})^{r-1}) $. 
			Note that $\Pr\left[d(v)\geq d_0-2\right]=\Pr\left[d(v)= d_0-2\right]+\Pr\left[d(v)\geq d_0-1\right]$. By Observation \ref{obs:BinFrac}, 
			$$\frac {\Pr[d(v)=d_0-2]}{\Pr[d(v)=d_0-1]}<\frac {2(d_0-2)n}{m}\leq 4\log^2 n,$$
			and
			$\Pr\left[d(v)\geq d_0-2\right]\leq (4\log^2n+1)\Pr[d(v)\geq d_0-1] \leq  20\log^4n\Pr[v\in L_0]$.\\
					
			Recall that by (the proof of) Observation~\ref{obs:multiplicity}, for $i\in [r-1]$, $\mu(v_iv_{i+1})\leq 2$ with probability $1-o(\frac 1{n^2})$.  
			Also, note that conditioning on the existence of the path $P$ (with multiplicities), the degree of $v_1$ (or $v_r$) outside $P$ is distributed according to $\Bin(m-\sum_{i=1}^{r-1}\mu(v_iv_{i+1}),p')$ where $p'$ is at most $\frac {n-2}{\binom n2-r+1}<\frac 2n$.
			 Then by Observation~\ref{obs:BinDec}, Remark~\ref{re:condi}, and Claim \ref{claim:PrL0}, 
			\begin{align*}
			\Pr[v_1,v_r\in L_0 \mid X_P]&=\Pr[d(v_1)\geq d_0 \mid X_P]\Pr[d(v_r)\geq d_0 \mid X_P,\ v_1\in L_0]\\
			&\leq \left(\Pr\left[d(v)\geq d_0-2\right]+o\left(\frac 1{n^2}\right)\right)^2\\
			&\leq\left(20\log^4n\Pr[v\in L_0]+o\left(\frac 1{n^2}\right)\right)^2\\
			&=O(n^{-1.8}\log^{12}n).
			\end{align*}
			Going over all choices of $v_1,\dots ,v_r$, the probability that there exists a path of length at most 2 with two endpoints in $L_0$ is at most \[\sum_{r=1}^{3}n^{r}\Pr[X_P]\Pr[v_1,v_r\in L_0 \mid X_P]=o(1).\qedhere \]
		\end{proof}

\medskip

We will show now that the set $L_0$ contains at least two vertices of different degrees.

\begin{claim}\label{claim:L0nonempty}	\whp there exist $v,u\in V$ such that $v,u\in L_0$ and $d(v)\neq d(u)$.
\end{claim}

\begin{proof}
	Since $L'_0\subseteq L_0$, it is enough to prove that \whp $|L_0|>|L'_0|>0$.
	First, by Claim~\ref{claim:PrL0}, we have that  
	$\Pr\left[v\in L'_0 \right]\leq n^{-0.95}\log^3n$, and $\Pr[v\in L_0]>n^{-0.9}$. Therefore, $\Pr[v\in L_0]=\omega\left(\Pr[v\in L'_0]\right)$. 
	Recall that by Claim~\ref{claim:L0Far} \whp every two vertices in $L_0$ are not connected. 
	
	Let $X\sim \Bin(m,\frac 2n)$ $X'\sim(m,\frac 2{n-1})$, and note that $\frac {2m}n\leq 2\log^{100}n$, thus by Lemma \ref{Che0} we have $\Pr[X>\log^{200}n]\leq {\rm{e}}^{-\log^{200}n}=o(n^{-0.95})$. 
	Let $(L,d)\in\{(L_0,d_0),(L'_0,d'_0)\}$
	and recall that $|L|=\sum_{v\in[n]}\textbf{1}_{d(v)\geq d}$. By Lemma~\ref{BinEq}, for  $s=\log^{200}n$ and $k=d$, we have that $\Pr[X'\geq d]=(1+o(1))\Pr[X\geq d]$.  Using Remark~\ref{re:condi} and Observation~\ref{obs:BinDec} we have that
	\begin{align*}
	\mathbb {E}(|L|^2) 
	&= \sum_{v,u\in[n]}\Pr[d(v)\geq d,\ d(u)\geq d]= \sum_{\substack{v,u\in[n]\\ u\neq v}}\Pr[d(v)\geq d,\ d(u)\geq d]+\mathbb{E}(|L|)\\
	& = \sum_{\substack{v,u\in[n]\\ u\neq v,\ u\nsim v}}\Pr[d(v)\geq d,\  d(u)\geq d]+\mathbb{E}(|L|)+o(1)\\  
	& = \sum_{\substack{v,u\in[n]\\ u\neq v,\ u\nsim v}}\Pr[d(u)\geq d]\Pr[d(v)\geq d\mid   d(u)\geq d]+\mathbb{E}(|L|)+o(1)\\
	& \leq n^2\Pr[X\geq d]\Pr[X'\geq d]+\mathbb{E}(|L|)+o(1)\\		
	& \leq (1+o(1))n^2\Pr[X\geq d]^2+\mathbb{E}(|L|)+o(1)\\
	& \leq (1+o(1))\mathbb E(|L|)^2+\mathbb{E}(|L|)+o(1).
	\end{align*} 
	All in all,  by Chebyshev inequality we have that
	$$\Pr\left[ \Bigl||L|-\mathbb E(|L|)\Bigr|>\frac 12\mathbb E(|L|)\right]\leq \frac {\text{Var}(|L|)}{\mathbb E(|L|)^2/4}\leq \frac {(1+o(1))\mathbb E(|L|)^2+\mathbb{E}(|L|)+o(1)-\mathbb E(|L|)^2}{\mathbb E(|L|)^2/4}=o(1).$$
	Therefore, \whp 
	$|L_0|\geq \frac 12\mathbb E(|L_0|)=\Omega(n^{0.1})$, 
	$|L'_0|\geq \frac 12\mathbb E(|L'_0|)=\Omega(n^{0.05})$,
	$|L'_0|\leq \frac 32\mathbb E(|L'_0|)=O(n^{0.05}\log^2n)$, and the claim follows.
\end{proof}

\medskip

In the next claim we show that \whp every two edges with multiplicity two are far from each other.

\begin{claim}\label{claim0:MultiDisjoint}
	Let $M\sim \mnm$. Then for every $3\leq r\leq 5$ and for every   $v_1,\dots,v_r\in [n]$, the probability that $\mu(v_1v_2),\mu(v_{r-1}v_r)\geq 2$, and $\mu(v_2v_3),\dots, \mu(v_{r-2}v_{r-1})\geq 1$ is $o(1)$. 
	In other words,  \whp for every pair of edges $e_1,e_2\in \binom{[n]}{2}$ with multiplicity at least 2, the distance between $e_1$ and $e_2$ is  at least three.
\end{claim}

\begin{proof}
%
Let $r$ be an integer. For $r$ vertices $v_1,\dots,v_r$ let $X_r$ be the event that there exist $i\neq j\in [r-1]$ such that $\mu(v_iv_{i+1})\geq 2$, $\mu(v_jv_{j+1})\geq 2$ and for every $k\in [r-1]\setminus \{i,j\}$, $\mu(v_kv_{k+1})\geq1$. Then by Claim~\ref{claim:edgeSet}
\begin{align*}
\Pr[X_r]\leq&(r-1)(r-2)\left(\frac {5{\rm{e}}m}{n^2}\right)^{4+(r-3)}\leq r^2\left(\frac {5{\rm{e}}m}{n^2}\right)^{r+1}.
\end{align*}
Therefore, the probability that there exist such $r$ vertices is at most $n^r\Pr[X_r]=O( \frac {\log^{100(r+1)}n}{n})$. Thus, the probability that there exist two edges $e,e'$ with $\mu(e),\mu(e')\geq 2$ and distance at most two is at most  $\sum_{r=3}^{5}n^r\Pr[X_r]=\sum_{r=3}^{5}O(\frac {\log^{100(r+1)}n}{n})=o(1)$.
\end{proof}

Now we are ready to prove Theorem \ref{GSforRM}  for the case $\frac n{\log\log n}\leq m\leq n\log^{100} n$.
\begin{proof}
	Let $M\sim \mnm$ for $m\leq n\log^{100} n$. Let $d_0=\max \{d \mid \Pr(\Bin(m,\frac 2n)\geq d)\geq n^{-0.9}\ \}$, and $L_0=\{v\in V\mid d(v)\geq d_0 \}$. Let $\Delta$ be the maximum degree in $M$. By Claim \ref{claim:L0nonempty} we have that \whp all vertices of degree $\Delta$ and $\Delta-1$ are in $L_0$. Recall also that by Observation \ref{obs:multiplicity}, \whp for every $e\in E(M)$, $\mu(e)\leq 2$. 
	Assume that these events occur and furthermore
	that the assertions of Claims \ref{claim:L0Far} and \ref{claim0:MultiDisjoint} hold. These
	events occur w.h.p., so from now on we argue deterministically and show that $\chi'(M)=\Delta$.
	
	 We remove edges from the multigraph in the following way.
	First we look at all $e$ such that $\mu(e)= 2$, denote this set by $E_1$ and  let $V_1=\cup_{e\in E_1}e$. We remove from each such $e$ one multiplicity. By Claim~\ref{claim0:MultiDisjoint}, we removed a matching $F_1$. The remaining multigraph is a simple graph.  Now, we look at all of the remaining vertices of degree $\Delta$, denote this set by $V_0$ and note that $V_0\cap V_1=\emptyset$. 
	 By Claim \ref{claim0:MultiDisjoint}, for every vertex $v$ with $d(v)\geq 2$  there exists a neighbour $u$ such that  $u\notin V_1$. Therefore, for every vertex $v\in V_0$ we find such neighbour $u$ and remove the edge $vu$.  By Claim \ref{claim:L0Far} we removed a matching $F_2$. Note that $F=F_1\cup F_2$ is also a matching. The remaining graph is a simple graph with maximum degree  $\Delta-1$. Furthermore, all of the current vertices of degree $\Delta-1$ are in $L_0$ and therefore by Claim \ref{claim:L0Far} they form an independent set. By Theorem \ref{vizing}, we can colour the edges of the graph with $\Delta-1$ colours. Colouring the edges of $F$ with the additional colour gives us the desired colouring.
\end{proof}

\vspace{3mm}
\subsection{The case $n\log^9 n\leq m\leq n^2\log^2n$ }\label{sec:m<n2logn}
In this section we will prove Theorem \ref{GSforRM}   and Theorem \ref{thm:nEven} for the case $n\log^{8} n\leq m\leq n^2\log^2 n$.

\begin{proof}
	By Observation \ref{obs:multiplicity}, we have that \whp  $\mu(M)\leq \log^3 n$. Also, by Lemma \ref{lem0:DegDiff}, for, say, $f(n)=\log\log n$, we have that \whp $d_1-d_2\geq \frac {1}{f(n)}\sqrt{\frac {2m}{n \log n}}$. 

	By Lemma \ref{lem:Chi'Tash} applied with $k=d_1$, it is enough to show that \whp $d_1-d_2\geq \max \{2,\mu(M)\}$. Clearly, $\frac {1}{f(n)}\sqrt{\frac {2m}{n \log n}}\geq 2$, and thus \whp $d_1-d_2\geq 2$. 
	For $m\geq n\log^9n$ we have that \whp $d_1-d_2\geq \frac {1}{f(n)}\sqrt{\frac {2m}{n \log n}}>\log^3n> \mu(M)$.
\end{proof}

\vspace{3mm}
\subsection{The case $m=\omega(n^2\log n)$, and $n$ is even}\label{sec:nEven}

In this section we complete the proof of Theorem \ref{thm:nEven}, and the proof of Theorem \ref{GSforRM} for the even case.

\begin{proof}
	  By Corollary~\ref{cor:DifMu}, we have that \whp $\mu(M)-\mu_{min}=O\left(\sqrt{\frac {2m}{n^2}\log n}\right)$.  
	  Also, by Lemma~\ref{lem0:DegDiff} we have that \whp
          $d_1-d_2\geq \frac {1}{f(n)}\sqrt{\frac {2m}{n \log n}}$,
          where $f(n)\to \infty $ arbitrarily slowly. That is, \whp
          $ d_1-d_2\geq \mu(M)-\mu_{min}$, and thus by Corollary~\ref{cor:RemoveCopiesKn} we have $\chi'(M)=\Delta(M)$.  
\end{proof}

%
%
%
%
%

\vspace{3mm}
\subsection{ The case $m=\omega(n^2\log n)$, and $n$ is odd}\label{sec:OddN}

In this section we prove Theorem \ref{GSforRM} for the case that $n$ is odd and $m$ is large.
We then deduce Corollary \ref{cor:oddN}.

Before proving Theorem \ref{GSforRM} and Corollary \ref{cor:oddN}, we will present the following observation that indicates the essential difference between small values of $m$ and large values of $m$, in the case that $n$ is odd. (It is the same phenomenon, as discussed in Section~\ref{intro}, that makes $\lceil\rho(G)\rceil$ a natural lower bound for $\chi'(G)$ for any multigraph $G$.)

\begin{observation}\label{obs:mIsLarge}
	Let $\varepsilon>0$ and let $m\geq (1+\varepsilon)n^3\log n$ where $n$ is odd. Then \whp $\Delta<\chi'(M)=(1+o(1))\Delta$.
\end{observation}

\begin{proof}
	Write $n=2k+1$. The maximal matching in $M$ is of size at most $k$, and therefore one needs at least $\frac mk$ colours to colour the edges of $M$. Thus $\frac mk \leq \chi'(M)$. Furthermore, by Claim~\ref{lem0:MaxDeg2} we have that \whp $\Delta\leq \frac {2m}n+\sqrt{(1+\varepsilon)\frac {2m}n\log (n^2)}$. Together we have that \whp $\chi'(M)-\Delta\geq \frac mk- \frac {2m}n-\sqrt{(1+\varepsilon)\frac {2m}n\log (n^2)}$. As $\sqrt m\geq n^{3/2}\sqrt {(1+\varepsilon)\log n}$, we get $\frac mk- \frac {2m}n-\sqrt{(1+\varepsilon)\frac {2m}n\log (n^2)}>0$, that is, \whp $\Delta<\chi'(M)$.
	
	For the upper bound,  Note that
	by Corollary~\ref{cor:DifMu}, we have that \whp $\mu(M)-\mu_{min}=O(\sqrt{\frac {2m}{n^2}\log n})$.  
	Also, by Lemma~\ref{lem0:DegDiff} we have that \whp
        $d_1-d_2\geq \frac {1}{f(n)}\sqrt{\frac {2m}{n \log n}}$. That
        is, \whp $d_1-d_2\geq \mu(M)-\mu_{min}$, and by Corollary~\ref{cor:RemoveCopiesKn} \whp \[\chi'(M)\leq\Delta(M)+\mu_{min}\leq (1+o(1))\Delta(M).\qedhere \]  
\end{proof}

For proving the theorems, we first need the following claim.
\begin{claim}\label{claim:rho}
	Let $m= \omega (n^2\log n)$, assume that $n$ is odd, and let $M\sim \mnm$. Then \whp $\rho(M)=\rho([n])=\frac m{\lfloor \frac {n}2 \rfloor}$.
\end{claim}

\begin{proof}
	First, note that since $n$ is odd then $\rho(M)$ is obtained on an odd set. Indeed, assume that $S\subset [n]$ and $|S|$ is even. Take $S'=S\cup \{v\}$ where $v\notin S$ (there is always such vertex since $S$ is even and $n$ is odd). Write $|S|=2k$, and so $|S'|=2k+1$. Then $\lfloor \frac {|S|}2 \rfloor=\lfloor \frac {|S'|}2 \rfloor$ and $e(S)\leq e(S')$. Thus $\rho(S)\leq \rho(S')$.
	
	Next, let $S\subsetneq [n]$ be an odd set. Write $n=2r+1$ and $|S|=2k+1$, then $2k+1<2r$. Note that
	
	\begin{align*}
	\rho([n]) &= \frac mr = \frac {e(S)+e([n]\setminus S)+e(S,[n]\setminus S)}r \\
	& \geq  \frac {e(S)}r +\frac {\binom {2r-2k}2 \cdot \mu_{\min}+ (2k+1)(2r-2k)\cdot \mu_{\min} }r \\ 
	& = \frac {e(S)}r +\frac { (r-k)(2r-2k-1) \cdot \mu_{\min}+ 2(2k+1)(r-k)\cdot \mu_{\min} }r \\
	& = \frac {e(S)}r +\frac { (2r-2k-1 + 2(2k+1))(r-k)\cdot \mu_{\min} }r \\
	& = \frac {e(S)}r +\frac { (2r+2k+1)(r-k)\cdot \mu_{\min} }r.  
	\end{align*}
	
	Thus it is enough to show that 
	$$\rho(S)=\frac {e(S)}k < \frac {e(S)}r +\frac { (2r+2k+1)(r-k)\cdot \mu_{\min} }r,$$
	that is, to show that 
	$ {e(S)}r <  {e(S)}k + { k(2r+2k+1)(r-k)\cdot \mu_{\min} }$, or simply $ {e(S)} <    { k(2r+2k+1)\cdot \mu_{\min} }$.
	
	Note that $e(S)\leq \binom{2k+1}2 \cdot \mu = (2k+1)k \cdot \mu$, where $\mu:=\mu(M)$. Therefore, we wish to show that $(2k+1)k \cdot \mu < k(2r+2k+1)\cdot \mu_{\min}$, that is, $\frac {\mu } {\mu_{\min} } < \frac {2r+2k+1}{2k+1} = 1+ \frac {2r}{2k+1}$. By Observation \ref{obs:multiplicity} and the fact that $m=\omega (n^2 \log n)$ we have that \whp $\frac {\mu} {\mu_{\min} }=1+o(1)$; on the other hand, we assumed that $2r>2k+1$, and thus the claim follows.
\end{proof}

We now ready to prove Theorem \ref{GSforRM} for the case that $n$ is odd and $m$ is large. In this proof, we will use Theorem \ref{nnlapp}.

\begin{proof}
	Write $\chi'(M)=k+1$ and let  $d_1\geq d_2\geq\ldots\geq d_n$ be the 
	degree sequence of $M$. By Theorem~\ref{nnlapp}, at least one of the items $(a),(b),(c),(d)$ holds. By Lemma \ref{lem0:DegDiff} and the fact that $m=\omega (n^2\log n)$, \whp $d_1-d_2=\omega(1)$. Therefore, if Item $(b)$ holds, we have that \whp $\chi'(M) \leq d_2+3=d_1-\omega(1)$, which is impossible since $d_1\leq \chi'(M)$ always. By Observation \ref{obs:multiplicity}, using $m=\omega(n^2\log n)$, we have that \whp $\mu_{\min}=(1-o(1))\mu$, therefore Item $(c)$ is also impossible.
	
	We are left with Items $(a)$ and $(d)$. If Item $(a)$ occurs, then $\chi'(M)\leq d_1$, which gives $\chi'(M)=d_1$. By Claim \ref{claim:rho}, \whp $\rho(M)=\rho([n])=\frac {e(M)}{\lfloor \frac n2 \rfloor}$. Since $n$ is odd we have that  $\lfloor \frac n2 \rfloor=\frac {n-1}2$, and thus if Item $(d)$ occurs, then \whp $\lceil\rho(M)\rceil=\lceil\rho([n])\rceil\geq k+1=\chi'(M)$. That is,  \whp $\lceil\rho(M)\rceil= \chi'(M)$ in the case of Item $(d)$.
\end{proof}

We conclude this section with the proof of Corollary~\ref{cor:oddN}.

\begin{proof}[Proof of Corollary  \ref{cor:oddN}]
	For the case $m\leq n^2\log^2 n$ we already saw in the previous subsections that \whp $\chi'(M)=\Delta(M)$, so we may assume that $m=\omega(n^2\log n)$.
	
	By Theorem \ref{GSforRM} we know that \whp $\chi'(M)=\max \{\Delta(M),\lceil\rho(M)\rceil \}$. By Claim \ref{claim:rho} we have \whp $\rho(M)=\frac {2m}{ {n-1}}$  and by (the proof of) Lemma \ref{lem0:MaxDeg2} we have that \whp $\Delta(M)\in \frac {2m}n+\sqrt {(1\pm \frac \varepsilon2)\frac {4m\log n}n}$. 
	
	Assume that $m\leq (1-\varepsilon)n^3\log n$, then \whp we have
	\begin{align*}
	\Delta(M) &\geq \frac {2m}n+\sqrt {\left(1- \frac \varepsilon2\right)\frac {4m\log n}n}\\
	&= \frac {2m}{n-1}\left( 1-\frac {1}n + \sqrt {\left(1-\frac \varepsilon2\right)\frac {(n-1)^2\log n}{nm} } \right)\\
	& \geq \frac {2m}{n-1}\left( 1-\frac {1}n + \sqrt {\left(1-\frac \varepsilon2\right)\frac {(n-1)^2\log n}{n(1-\varepsilon)n^3\log n } } \right)\\
	& = \frac {2m}{n-1}\left( 1-\frac {1}n + \frac {n-1}{n^2}\sqrt{1+\frac \varepsilon{2(1-\varepsilon)}} \right)>  \frac {2m}{n-1}\left(1+\frac {\Theta_\varepsilon(1)} n\right)\\
	&> \frac {2m}{n-1}+1\geq \lceil\rho(M)\rceil,
	\end{align*}
	where 
	$m=\omega(n^2\log n)$.
	
	Assume now that $m\geq (1+\varepsilon)n^3\log n$, then by Observation~\ref{obs:mIsLarge} we have that \whp $\Delta(M)<\chi'(M)$, and since \whp $\chi'(M)=\max \{\Delta(M),\lceil\rho(M)\rceil \}$ we are done.
\end{proof}
%

\vspace{3mm}
\subsection{Algorithmic issues}\label{sec:alg}
We conclude this section with a discussion about algorithmic aspects
of our results and proofs. We begin by noting that Vizing's theorem (Theorem~\ref{vizing}) is
algorithmic, and could be stated as follows.  

\begin{theorem}[Vizing \cite{vizing64}]\label{vizingAlg}
	There exists a polynomial-time algorithm that takes as input any simple graph
        $G$ with the property that every cycle of $G$ contains  a vertex of degree less than
        $\Delta(G)$, and returns an edge-colouring of $G$ with
        $\Delta(G)$ colours.
\end{theorem} 

Using this, we get that the proof of Theorem~\ref{thm:Gnp} is algorithmic, as well as the proof of Theorem~\ref{GSforRM} for the case $m\leq n\log^{100}n$.

We will next (briefly) explain how to show the existence of an
efficient (polynomial time) algorithm that \whp finds an optimal
edge-colouring for other values of $m:=m(n)$. In
Remark~\ref{re:TashAlg} we mentioned
that the proof of Theorem~\ref{TashT} is algorithmic. For convenience
let us say
$(G,k,e_0,\phi,T)$ is a {\it legal} 5-tuple if it satisfies the
conditions of Theorem~\ref{TashT}, namely:
\begin{itemize}
\item $k$ is a positive integer,
\item $G$ is a multigraph with $\Delta(G)\leq k$,
\item $e_0$ is an edge of $G$ such that the endpoints $x$ and
    $y$ of $e_0$ satisfy $d(x)+d(y)\leq2k-2$,
\item    $d(v)\leq k-1$ for each $v\in V(G)\setminus\{x,y\}$,
\item $\phi$ is a $k$-edge-colouring of $G-e_0$,
\item $T$ is a $\phi$-Tashkinov tree starting with $e_0$.
\end{itemize}
Then Theorem~\ref{TashT} could be restated as follows.

\begin{theorem}\label{TashAlg}
There exists an algorithm $Alg1$ that takes as input any legal 5-tuple
$(G,k,e_0,\phi,T)$ with the property that $T$ is not $\phi$-elementary,
and returns in polynomial time a $k$-edge-colouring $\psi$ of $G$.
\end{theorem}
The colouring $\psi$ is obtained from $\phi$ by a series of
alternating path switches (see e.g. the proof of Theorem~5.1 in
~\cite{S} and
associated discussion of its algorithmic aspects). 

Using Alg1, we obtain an algorithm AlgC that finds either an
edge-colouring or an elementary Tashkinov tree. Here we say a multigraph $G$ is
$(k,t)$-{\it bounded} if $\Delta(G)\leq k$ and $d(v)\leq k-t$ for all
except possibly one vertex of $G$.

\begin{theorem}\label{ColAlg}
There exists an algorithm $AlgC$ that takes as input any positive
integer $k\geq 2$ and any $(k,2)$-bounded multigraph $G$, and returns
in polynomial time either
\begin{enumerate}
\item a $k$-edge colouring of $G$, or
\item a legal 5-tuple
$(G^*,k,e_0,\phi,T)$ where $G^*$ is a
  subgraph of $G$ and $T$ is a maximal
$\phi$-elementary Tashkinov tree.
\end{enumerate}
Moreover, if the subgraph $G^*$ in output (2) has a vertex $w$ with
$d(w)\geq k-1$ (for example $G^*=G$ is possible), then the edge $e_0$ is incident to $w$.
\end{theorem}

\begin{proof}
AlgC proceeds as follows. Given $G$ and $k$, list the edges of $G$ in an
arbitrary order, except that if $G$ has a vertex $w$ with $d(w)\geq
k-1$ then end the
list with the set of edges incident to $w$. Colour the
edges greedily in order with $k$ colours until they are all coloured (and so we are done) or
there is no available colour for the current edge. In this case call
the current edge $e_0$, remove all other uncoloured edges and call the
resulting subgraph $G'$. With the
current partial colouring $\phi$ of $G'-e_0$, grow
a maximal Tashkinov tree $T$ from $e_0$ (this is easily seen to be a
polynomial procedure). If $V(T)$ is not 
$\phi$-elementary then $(G',k,e_0,\phi,T)$ is a legal
5-tuple to which Theorem~\ref{TashAlg} applies. Apply Alg1 to obtain
a $k$-edge-colouring of $G'$. Restore the remaining uncoloured edges
in the order and
repeat this last step with the next edge in the order until either all of
$G$ is coloured, or we stop at a
partial colouring $\phi$ of some subgraph $G^*$ of $G$, and a
maximal $\phi$-elementary Tashkinov tree $T$ in $G^*$. Then
$(G^*,k,e_0,\phi,T)$ is a legal 5-tuple as in output (2). Moreover if
$G^*$ contains a vertex of degree at least $k-1$ then it is $w$, and
by definition of the edge order $e_0$ is incident to $w$.
\end{proof}

\medskip

For the case that $n\log^9n\leq m\leq n^2\log^2n$, and the case that
$m=\omega(n^2\log n)$ where $n$ is
 even, we may use the following algorithmic version of
 Lemma~\ref{lem:Chi'Tash}. 

\begin{lemma}\label{lem:Chi'TashAlg}
The algorithm AlgC returns a $k$-edge-colouring of $G$, for any
$(k,t)$-bounded multigraph $G$ where $t\geq\max\{2,\mu(G)\}$.
\end{lemma}

\begin{proof}
Suppose on the contrary that AlgC does not return a
$k$-edge-colouring when applied to $G$. Then instead it returns a legal 5-tuple
$(G^*,k,e_0,\phi,T)$ where $G^*$ is a
subgraph of $G$ and $T$ is a maximal
$\phi$-elementary Tashkinov tree. Let $w$ denote one endpoint of
$e_0$, where (as given by AlgC) we may assume that $w$ is the unique
vertex in $G^*$ of degree at least $k-1$, if it exists.

As in the proof of Lemma~\ref{lem:Chi'Tash} applied to $G^*$, with
$e_0$, $\phi$ and $T$ given by AlgC, we find
that $d_T(w)\geq t(|V(T)|-1)+1$. But
clearly $d_T(w)\leq\mu(|V(T)|-1)$, so we find $\mu>t$, contradicting
the assumption of the lemma.
\end{proof}

Having Lemma~\ref{lem:Chi'TashAlg} (instead of
Lemma~\ref{lem:Chi'Tash}) we can simply go over the proofs
in Section~\ref{sec:m<n2logn} and in Section~\ref{sec:nEven} to
obtain an efficient algorithm that \whp colours the edges of $M\sim M(n,m)$ with
$\max\{\Delta(M),\lceil\rho(M)\rceil \}$ colours.
\medskip

Finally we give only a brief outline of what happens in the case
$m=\omega(n^2\log n)$ and $n$ is odd.
In a very similar way as above (but with more technical details), for
$k=\max\{\Delta(G),\lceil\rho(G)\rceil \}$ we show that AlgC
will succeed in finding a $k$-edge-colouring of a multigraph $G$ unless one of
$(b)$ or $(c)$ in Theorem~\ref{nnlapp} holds (here $(a)$ and $(d)$ will not hold, by
definition of $k$). The argument involves showing that if on the
contrary output (2) is the result when AlgC is applied to $G$, then
sets $X$, $Z$, $Q$ and $U$ as in the statement of
Lemma~\ref{nnl} can be constructed. (In particular $Q$ is the vertex
set of the Tashkinov tree $T$ given by output (2).) Then following the
proof of Theorem~\ref{nnlapp} tells us that one of $(b)$ or $(c)$
holds. We may then conclude that \whp AlgC succeeds in finding a
$k$-edge-colouring of $M\sim M(n,m)$ when  $m=\omega(n^2\log n)$ and
$n$ is odd, since \whp $(b)$ does not hold  according to
Lemma~\ref{lem0:DegDiff}, \
and \whp $(c)$ does not hold according to Observation~\ref{obs:multiplicity}.

\vspace{5mm}
\section{Concluding remarks and further discussion}\label{sec:con}

In this paper we considered the chromatic index of random multigraphs.
  We showed that random multigraphs are \whp first class, verifying the Conjecture~\ref{conj:GS} for the random multigraph case.   In particular, in Section~\ref{sec:tash} we gave general tools for bounding the chromatic index of a multigraph (see Theorem~\ref{nnlapp} and Lemma~\ref{lem:Chi'Tash}). We also addressed the informal conjecture mentioned in \cite{S} that almost all multigraphs are first class. To be more formal, Theorem~\ref{GSforRM} proves that if $M\sim M(n,m)$ and $\Pr(M)$ is the probability that $M$ was sampled according to the distribution $M(n,m)$, then $\lim_{n\to \infty}\sum_{\chi'(M)=\max\{\Delta,\lceil\rho\rceil \} }\Pr(M)=1$.
  
   Another natural way to generalize the definition of the random graph
   model $G(n,m)$ to multigraphs is to give each multigraph on $n$
   vertices and $m$ edges the same probability (that is, to sample a
   multigraph with $m$ edges and $n$ vertices uniformly at random). 
   Note that this probability space is not equivalent to the model
   $M(n,m)$ studied
   in this paper. Indeed, in $M(n,m)$
   multigraphs that have ``balanced" multiplicities are chosen with
   larger probabilities, unlike the uniform model that gives the same
   probability to a multigraph with one edge with multiplicity $m$,
   and a multigraph with many edges and smaller multiplicities. Note in
   addition that, similarly to the random graph model $G(n,p)$,
   our model gives a product probability space. 
   For these reasons we found the model $M(n,m)$ to be more natural to consider.
   However, we believe that using similar methods to those used in
   this paper, one can show that the Conjecture~\ref{conj:GS}
   holds also in the uniform model. 
  
  Note that our proofs rely heavily on the degree gap, and thus one
  cannot use the same techniques to prove similar results for
  multigraphs without this property. Therefore,  it would be
  interesting to study the same problem in the context of $d$-regular
  multigraphs in the random case, or for
  any family of nearly regular multigraphs. We expect this to be a challenging task, as for example the typical value of the chromatic index of a random $d$-regular graph $G_{n,d}$, for non-constant $d$, has been established only very recently by Ferber and Jain \cite{FerberJain}.  
  
   An additional (simpler) random multigraph model one can consider is a model corresponding to the random graph model $G(n,p)$. Let $M_p(n,m)$ be the probability space of all multigraphs on $n$ vertices where each pair $e\in \binom{[n]}2$  appears independently with multiplicities $\mu(e)\sim Poisson(\lambda)$, where $\lambda=\frac m{\binom n2}$. This way, the expected number of edges of the multigraph is $m$, and for a vertex $v$ we have $d(v)\sim Poisson(\frac {2m}{n})$.
Given the well-known connection between the binomial  and Poisson distributions, we can expect results similar to those obtained above to hold for this model as well. \\

Another related question one can also study concerns the list chromatic index in the context of random multigraphs. The \textit{list chromatic index} of a multigraph $M$, denoted by $\chi'_\ell$, is the smallest integer $k$ such that for every assignment of lists of size $k$ to each edge, there exists an edge-colouring such that each edge receives a colour from its list. Clearly, $\chi'(M)\leq \chi'_\ell(M)$, and a trivial upper bound for the list chromatic index is $\chi'_\ell< 2\Delta(M)$. Borodin, Kostochka, and Woodall \cite{BKW} showed that in fact $\chi'_\ell(M)\leq \left\lfloor 3\Delta(M)/2 \right\rfloor$, improving the previously known best upper bound of $9\Delta(M)/5$ due to Hind \cite{Hind}. The famous list colouring conjecture states that for every multigraph $M$ one has $\chi'(M)=\chi'_\ell(M)$ (see \cite{S}, p.260 for a short discussion). This conjecture was verified for the case that $M$ is a bipartite multigraph by Galvin \cite{Galvin}. Later on, Kahn proved \cite{KahnList} that $\chi'(M)=(1+o(1))\chi'_\ell(M)$.
The presently rather limited understanding of  edge colouring problems for multigraphs appears to indicate that the list colouring conjecture is very difficult.  
For this reason, we believe that attacking this problem in the random multigraph setting could be fruitful, as in this case we found the exact form of the chromatic index. 

\subsection*{Acknowledgements}
The authors would like to thank the referees of the paper for their careful reading and  helpful remarks.

%
%

\end{document}